\documentclass[12pt]{article}
\usepackage[utf8]{inputenc}
\usepackage{mathtools}                                  
\usepackage{booktabs}
\usepackage[english]{babel} % English language/hyphenation
\usepackage{amsmath,amsthm}
\usepackage{ amssymb }
\usepackage{graphicx}
\usepackage{array}
\usepackage{multirow}
\usepackage{hhline}
\usepackage[titletoc]{appendix}

\usepackage[margin=1in,hmarginratio=1:1,top=20mm,columnsep=20pt]{geometry} % Document margins
\usepackage{booktabs} % Horizontal rules in tables
\usepackage{natbib}
\usepackage{setspace}

\usepackage{graphicx}
\usepackage{subfigure}

\usepackage{footnote}
\newtheorem{lemma}{Lemma}[section]
\newtheorem{defn}{Definition}[section]
\newtheorem{assum}{Assumption}[section]
\newtheorem{prop}{Proposition}[section]

\newtheorem{remark}{Remark}[section]
\newtheorem{theorem}{Theorem}[section]

\newtheorem{condition}{Condition}[section]

\newcommand{\RR}{\mathbb{R}}

\theoremstyle{definition}
\newtheorem{expl}{Example}

\newcommand{\argsup}{\operatorname*{argsup}}

\begin{document}

\title {Gaussian Processes with Input Location Error and Applications to the Composite Parts Assembly Process}
\author{Wenjia Wang\\Hong Kong University of Science and Technology \\ \and Xiaowei Yue\\Virginia Polytechnic Institute and State University\\ \and Benjamin Haaland\\ University of Utah\\ \and C. F. Jeff Wu\\Georgia Institute of Technology} 
\date{}% do not change this line
\maketitle
\vspace{-5mm}

\begin{abstract}
In this paper, we investigate Gaussian process modeling with input location error, where the inputs are corrupted by noise. Here, the best linear unbiased predictor for two cases is considered, according to whether there is noise at the target unobserved location or not. We show that the mean squared prediction error converges to a non-zero constant if there is noise at the target unobserved location, and provide an upper bound of the mean squared prediction error if there is no noise at the target unobserved location. 
% We show that the mean squared prediction error does not converge to zero in either case.
We investigate the use of stochastic Kriging in the prediction of Gaussian processes with input location error, and show that stochastic Kriging is a good approximation when the sample size is large. Several numeric examples are given to illustrate the results, and a case study on the assembly of composite parts is presented. Technical proofs are provided in the Appendix.
\end{abstract}

\section{Introduction}
Gaussian process modeling is widely used to recover underlying functions from scattered evaluations, possibly corrupted by noise. This method has been utilized in spatial statistics for several decades \cite{cressie2015statistics,matheron1963principles}. Later, Gaussian process modeling has been applied in computer experiments to build emulators of their outputs \cite{sacks1989design}. In order to capture the randomness of real systems, it is natural to use stochastic simulation in computer experiments. For Gaussian process modeling, the output associated with each input can be decomposed as the sum of a mean Gaussian process output and random {i.i.d. }(Gaussian) noise. Following the terminology in design of experiments \cite{wu2011experiments}, we call the noise added to the mean Gaussian process output as \textit{extrinsic} noise. The extrinsic noise is usually from uncertainty associated with responses, such as measurement errors, computational errors and other unquantified errors, {and does not come from random process}. The corresponding Gaussian process modeling with extrinsic noise is called \textit{stochastic Kriging} \cite{ankenman}. In spatial statistics, the extrinsic noise {attributes to the nugget effect} \cite{matheron1963principles,roustant2012dicekriging,stein1999interpolation}.

Besides extrinsic noise, in some cases, the input variables are also corrupted by noise. Noisy or uncertain inputs are quite common in spatial statistics, because geostatistical data are often indexed by imprecise locations. Detailed examples can be found in \cite{barber2006modelling,veneziano1987statistical}. We call the noise of input variables as \textit{intrinsic} noise. {It is worth noting that the ``intrinsic'' in this paper is different from the ``intrinsic'' for stochastic processes. In this paper, the intrinsic noise comes from the natural uncertainties inherent to the complex systems, such as actuating uncertainty, controller fluctuation, and internal measurement error. In contrast to the extrinsic noise that is related to the response, intrinsic noise is associated with input variables.} If the input variables are corrupted by noise in a Gaussian process, it is known as a Gaussian process with input location error, and the corresponding best linear unbiased predictor is called Kriging adjusting for location error (KALE) \cite{cressie2003spatial}. Also see \cite{bocsi2013hessian,dallaire2009learning,girard2004approximate,mchutchon2011gaussian} for more discussions. KALE has been applied in a spectrum of arenas, including robotics \cite{deisenroth2015gaussian}, wireless networks \cite{muppirisetty2016spatial}, and Wi-Fi fingerprinting \cite{he2017indoor}. 

KALE predicts the mean Gaussian process output at an unobserved point \textit{without} intrinsic noise. In many applications, however, the prediction of the mean Gaussian process output at an unobserved point \textit{with} intrinsic noise is desired. A motivating example is the composite aircraft fuselage assembly process. In this process, a model is needed to predict the dimensional deviations under noisy actuators' forces. Further, when new actuator forces are implemented in practice, there is an inevitable intrinsic noise, i.e., uncertainty in the actually delivered actuator forces. Therefore, the output at an unobserved point has intrinsic noise. Under this scenario, we consider Kriging adjusting for location error and noise (KALEN), which is the best linear unbiased predictor of the mean Gaussian process output at an unobserved point with intrinsic noise. {For another example, in the electric stability control system of vehicles, a model is developed to link the inputs (i.e., braking pressure and engine torque) and the outputs (i.e., stability control loss). Intrinsic noise inevitably exists in this system due to the uncertainties in wheel pressure modulators, pressure reservoir, and electric pump. Other than the two example mentioned above, KALEN fits many applications better than KALE due to the ubiquity of actuating errors in engineering systems.} %{\bf WW: Could Xiaowei add more industrial examples here? Saying that KALEN is ubiquitous in industry.}

In this paper, we discuss three predictors, KALE, KALEN, and stochastic Kriging, applied in prediction and uncertainty quantification of Gaussian process modeling with input location error. We show that unlike Gaussian process modeling without location error, the mean squared prediction error (MSPE) of these three predictors does not converge to zero as the sample size goes to infinity. Furthermore, we show that the limiting MSPE of KALEN and stochastic Kriging are equal if an unobserved point has intrinsic noise. We obtain an asymptotic upper bound on the MSPE of KALE and stochastic Kriging if there is no noise at an unobserved point. {This upper bound is small if the intrinsic noise at observed points is small.} Numeric results indicate that if the sample size is relatively small and noise is relatively large, KALE or KALEN have a much smaller MSPE, and thus are desirable, compared with stochastic Kriging. {If the sample size is large or the noise is quite small, then the performance of all three approaches is similar.} We also compare the performance of KALEN and stochastic Kriging in the modeling of a composite parts assembly process problem. We find that the KALEN and stochastic Kriging are comparable across a range of small intrinsic noise levels, corresponding to a range of actuator tolerances, which is consistent with the theoretical analysis.  

The remainder of this article is structured as follows. In Section \ref{GPwithLE}, we formally state the problem, introduce KALE and KALEN, and show some asymptotic properties of the MSPE of KALE and KALEN. Section \ref{sec:CompareNugg} presents some theoretical results when using stochastic Kriging in the prediction of Gaussian processes with input location error. Parameter estimation methods are discussed in Section \ref{sec:ParaEst}. Numeric results are presented in Section \ref{secNumerical}. A case study of the composite parts assembly process is considered in Section \ref{sec:casestudy}. Technical details are given in the Appendix.

\section{Gaussian Processes with Input Location Error}\label{GPwithLE}
In this section, we introduce two predictors of Gaussian processes with input location error, KALE and KALEN. We also give several asymptotic properties of KALE and KALEN.

\subsection{Two Predictors of Gaussian Processes with Input Location Error}\label{sec:KALEmodel}

Suppose $f$ is an underlying function defined on $\mathbb{R}^d$, and the values of $f$ on a convex and compact set $\Omega$ are of interest. Suppose we observe the responses $f(x_1),\ldots, f(x_n)$ on $X = \{x_1,\ldots, x_n\}\subset \Omega$. Following the terminology in design of experiments \cite{wu2011experiments}, we call $X = \{x_1,\ldots, x_n\}$ design points. A standard tool to build emulators based on observed data is Gaussian process modeling (see \cite{fang2005design} and \cite{santner2003design}, for example).  {In Gaussian process modeling, the underlying function $f$ is assumed to be a realization of a Gaussian process $Z$. From this point of view, we shall not differentiate $f$ and $Z$ in the rest of this work. We suppose $f$ is \textit{stationary}, which means that the covariance of $f(x)$ and $f(x')$ depends only on the difference $x-x'$ between the two input variables $x$ and $x'$. We further assume $\text{Cov}(f(x),f(x'))=\sigma^2 \Psi(x - x'),$ where $\sigma^2$ is the variance, and $\Psi$ is the correlation function. Then $\Psi$ should be positive definite and satisfy $\Psi(0)=1$. In Gaussian process modeling, one can assume that the mean of $f$ is zero, a constant, or a linear combination of known functions. The corresponding methods are referred to as simple Kriging, ordinary Kriging, and universal Kriging, respectively. Ordinary Kriging and universal Kriging are more flexible and may improve the prediction performance, but the estimation of the mean function can introduce more uncertainties. Moreover, Theorem 3 of \cite{wang2019prediction} suggests that the estimation of the mean function can be inconsistent. These uncertainties and inconsistency make the theoretical analysis more cumbersome, and dilute the focus of the overall analysis. Therefore, for ease of mathematical treatment,} we assume the mean of $f$ is zero in theoretical development, which is equivalent to removing the mean surface. {Nevertheless, we use a non-zero mean function in numeric studies and case study to improve the prediction performance by introducing more degrees of freedom.} 

For a Gaussian process with input location error, the inputs are corrupted by noise. In this paper, we mainly focus on the intrinsic error and assume the responses are not influenced by the extrinsic error. {It is worth noting that this assumption can be relaxed, and the Gaussian process with both intrinsic error and extrinsic error can be analyzed in a similar manner, as stated in Remark \ref{rem:rem1}.} Specifically, suppose the responses are perturbed by the intrinsic error, that is, we observe {$y_j = f(x_j+\epsilon_j)$} for $x_j\in X$, where the $\epsilon_j$'s are i.i.d. random {vectors with mean $0$, and have a probability density function $p(\cdot)$. It is possible to have replicates on some design points, i.e., for some $j\neq k$, $x_j = x_k$ for $x_j,x_k\in X$ but $\epsilon_j\neq \epsilon_k$. 
% Let $\bar X=\{\bar x_1,...,\bar x_m\}$ denote distinct design points, and $n_1,...,n_m$ be the number of replicates on each distinct design point. Thus, $\sum_{j=1}^m n_j = n$. 
We assume $p(\cdot)$ is continuous and each element of $\epsilon_j$ has finite variance.}

Following the approach in \cite{cressie2003spatial}, the best linear unbiased predictor of $f(x)$ on an unobserved point $x$ is given by
\begin{align}\label{pyx}
    Q(Y;x) = \alpha_1^TY + \alpha_2,
\end{align}
where $\alpha_1 \in \mathbb{R}^n,\alpha_2\in \mathbb{R}$ are the solution to the optimization problem
\begin{align}\label{mintoblup}
   \min_{(\alpha_1,\alpha_2)} \mathbb{E}(f(x) -  Q(Y;x))^2 = \min_{(\alpha_1,\alpha_2)} \mathbb{E}(f(x) -  \alpha_1^TY - \alpha_2)^2,
\end{align}
and the responses on the design points are $Y= (y_1,\ldots,y_n)^T$. By minimizing \eqref{mintoblup} with respect to $(\alpha_1,\alpha_2)$, we obtain the solution to \eqref{mintoblup} is $\alpha_1 = K^{-1}r(x)$ and $\alpha_2 = 0$, where $r(x) = (r(x,x_1),\ldots,r(x,x_n))^T$ denotes the covariance vector between $f(x)$ and $Y$ with  
\begin{align}\label{eq:RvectorNoi}
r(x,x_j) =\mathbb{E}(f(x)y_j) = \sigma^2\int\Psi(x-(x_j+\epsilon_j))p(\epsilon_j)d\epsilon_j,
\end{align}
and $K=(K_{jk})_{jk}$ denotes the covariance matrix with 
\begin{align}\label{eq:KernelMatrixNoi}
K_{jk} = \mathbb{E}(y_jy_k)=\left\{\begin{array}{ll} 
\sigma^2\Psi(x_j-x_j), & j=k, \\ 
\sigma^2\iint\Psi(x_j+\epsilon_j-(x_k+\epsilon_k))p(\epsilon_j)p(\epsilon_k)d\epsilon_j d\epsilon_k, & j\neq k.\end{array}\right.
\end{align} 
Plugging $\alpha_1 = K^{-1}r(x)$ and $\alpha_2 = 0$ into \eqref{pyx}, we find the best linear unbiased predictor of $f(x)$ is 
\begin{align}\label{BLUPNoN}
\hat f(x) = r(x)^TK^{-1}Y.
\end{align}
{\begin{remark}\label{rem:rem1}
If the observations also have i.i.d. distributed extrinsic noise with mean zero and finite variance $\sigma^2_\delta$, we only need to replace $\mathbb{E}(y_jy_j) = \sigma^2\Psi(x_j-x_j)$ by $\mathbb{E}(y_jy_j) = \sigma^2\Psi(x_j-x_j) + \sigma^2_\delta$, and the rest of the theoretical analysis remains similar. Our theoretical analysis can also be generalized to the case that $\epsilon_i$'s are independent but not identically distributed. Although these generalizations do not influence the theoretical development a lot, they could dilute the main focus of this paper. Therefore, we focus on the Gaussian processes with only i.i.d. intrinsic noise.
\end{remark}}

In \cite{cressie2003spatial} equation (\ref{BLUPNoN}) is referred to as Kriging adjusting for location error (KALE). If the prediction of $y(x)$ on an unobserved point $x$ with intrinsic noise is of interest, it can be shown that we only need to replace $r(x)$ in (\ref{BLUPNoN}) by $r_N(x)= (r_N(x,x_1),\ldots,r_N(x,x_n))^T$, where 
\begin{align}\label{eq:RvectorNoiX}
r_N(x,x_j) =\sigma^2\iint\Psi(x+\epsilon-(x_j+\epsilon_j))p(\epsilon_j)p(\epsilon)d\epsilon_jd\epsilon.
\end{align}
We refer to the corresponding best linear unbiased predictor $\hat y(x) = r_N(x)^TK^{-1}Y$ as Kriging adjusting for location error and noise (KALEN). One direct relation between KALE and KALEN is $\hat y(x) = \int \hat f(x+\epsilon) p(\epsilon) d\epsilon$.

In some cases, there exist closed forms of the integrals in (\ref{eq:RvectorNoi})--(\ref{eq:RvectorNoiX}). For example, if the correlation function $\Psi(s-t) = \exp(-\theta\|s - t\|_2^2)$, and the noise $\epsilon \sim N(0,\sigma_\epsilon^2I_d)$, where $\theta > 0$ is the correlation parameter, and $N(0,\sigma_\epsilon^2I_d)$ is a mean zero normal distribution with covariance matrix $\sigma_\epsilon^2I_d$, then (\ref{eq:RvectorNoi})--(\ref{eq:RvectorNoiX}) can be calculated respectively as \cite{cervone2015gaussian}
\begin{align}\label{eq:KernelMnormal}
K_{jk}=\left\{\begin{array}{ll} \sigma^2 & j=k, \\ 
\frac{\sigma^2}{(1+4\sigma^2_\epsilon\theta)^{d/2}}e^{\frac{-\theta\|x_j-x_k\|_2^2}{1+4\sigma^2_\epsilon\theta}} & j\neq k,\end{array}\right.\nonumber\\
r(x,x_j)=\frac{\sigma^2}{(1+2\sigma^2_\epsilon\theta)^{d/2}}e^{\frac{-\theta\|x-x_j\|_2^2}{1+2\sigma^2_\epsilon\theta}},\nonumber\\
r_N(x,x_j)=\frac{\sigma^2}{(1+4\sigma^2_\epsilon\theta)^{d/2}}e^{\frac{-\theta\|x-x_j\|_2^2}{1+4\sigma^2_\epsilon\theta}}.
\end{align}
{We also include the calculation of \eqref{eq:KernelMnormal} in Appendix \ref{app:eq8} for readers' reference.}

Unfortunately, in general, equations (\ref{eq:RvectorNoi})--(\ref{eq:RvectorNoiX}) are intractable and need to be calculated via Monte Carlo integration by sampling $\epsilon_j$'s from $p(\cdot)$, which can be computationally expensive. For example, if we choose the Mat\'ern correlation function, then \eqref{BLUPNoN} does not have a closed form. In this case, the calculation of \eqref{BLUPNoN} will require much time, as we will see in Section \ref{secNumerical}. 

\subsection{The Mean Squared Prediction Error of KALE and KALEN}\label{MSPE}

Now we consider the mean squared prediction error (MSPE) of KALE and KALEN. The MSPE of KALE can be calculated by
\begin{align}\label{eq:MSPEKALE}
\mathbb{E}(f(x) - \hat{f}(x))^2 & = \mathbb{E}(f(x) - r(x)^TK^{-1}Y)^2\nonumber\\
& = \mathbb{E}(f(x)^2) - 2r(x)^TK^{-1}\mathbb{E}(f(x)Y)+ r(x)^TK^{-1}\mathbb{E}(YY^T)K^{-1}r(x)\nonumber\\
& = \sigma^2\Psi(x-x) - r(x)^TK^{-1}r(x),
\end{align}
where $\hat{f}$ is as in \eqref{BLUPNoN}, and $r$ and $K$ are as defined in (\ref{eq:RvectorNoi}) and (\ref{eq:KernelMatrixNoi}), respectively. The last equality is true because of \eqref{eq:RvectorNoi} and \eqref{eq:KernelMatrixNoi}. Note that $\Psi(x-x) = \Psi(0)=1$. Similarly, one can check the MSPE of KALEN is
\begin{align}\label{eq:MSPEKALEN}
\mathbb{E}(y(x) - \hat{y}(x))^2 = \sigma^2\Psi(x-x) - r_N(x)^TK^{-1}r_N(x),
\end{align}
where $r_N$ is as defined in (\ref{eq:RvectorNoiX}).

Define 
\begin{align}\label{eq:newKernel}
\Psi_S(s-t) = \iint\Psi(s+\epsilon_1-(t+\epsilon_2))p(\epsilon_1)p(\epsilon_2)d\epsilon_1 d\epsilon_2.
\end{align}

In Proposition 3.1 of \cite{cervone2015gaussian}, it is shown that if a function $c(s,t) = \Psi_S(s-t)$ for $s\neq t$ and $c(s,s) = \Psi(s-s)$, then $c(\cdot,\cdot)$ is a valid correlation function. Therefore, the covariance matrix $K$ defined in (\ref{eq:KernelMatrixNoi}) is positive definite. We first consider the asymptotic properties of \eqref{eq:MSPEKALEN} as the fill distance goes to zero, where the fill distance $h_X$ of the design points $X$ is defined by
\begin{equation}\label{filldistance}
h_{X}:=\sup_{x\in\Omega}\min_{x_j\in X}\| x- x_j\|_2.
\end{equation}
Notice that the MSPE of KALEN can be expressed as
\begin{align}\label{decomex}
\mathbb{E}(y(x) - \hat{y}(x))^2 & = \sigma^2\Psi(x-x) - r_N(x)K^{-1}r_N(x)\nonumber\\
                       & = \sigma^2(\Psi(x-x) - \Psi_S(x-x)) + \sigma^2\Psi_S(x-x) - r_N(x)K^{-1}r_N(x).
\end{align}
Let $K_S = \sigma^2(\Psi_S(x_j-x_k))_{jk}$. Thus, $K = K_S + \sigma^2(\Psi(x-x) - \Psi_S(x-x))I_n$, where $I_n$ is an identity matrix. In the rest of Section \ref{GPwithLE} and Section \ref{sec:CompareNugg}, we assume the correlation function $\Psi$ satisfies the following assumption.
\begin{assum}\label{assumpsi}
The correlation function $\Psi$ is a radial basis function, i.e., $\Psi(s-t) = \phi(\|s-t\|_2)$ for $s,t\in \Omega$. Furthermore, $\phi(r)>0$ is a strictly decreasing function of $r\in \mathbb{R}^+$, with $\phi(0) = 1$. {The reproducing kernel Hilbert space generated by $\Psi$ can be embedded into a Sobolev space $H^m(\Omega)$ with $m> d/2$.}
\end{assum}
{
\begin{remark}
For a brief introduction to the reproducing kernel Hilbert space, see Appendix \ref{app:introrkhs}. 
\end{remark}
}

Many widely used correlation functions, including isotropic Gaussian correlation functions and isotropic Mat\'ern correlation functions, satisfy this assumption. {See Appendix \ref{app:introrkhs} for details.} For anisotropic correlation functions that have form $\Psi(s-t) = \phi(\|A(s-t)\|_2)$ with $A$ a diagonal positive definite matrix and $s,t\in \Omega$, we can stretch the space $\Omega$ to $\Omega'$ such that $\Psi(s'-t') = \phi(\|s'-t'\|_2)$ for $s',t'\in \Omega'$. Assumption \ref{assumpsi} implies $\Psi_S(x-x)< \Psi(x-x)$. Intuitively $K$ is equal to a covariance matrix plus a nugget parameter {equal to $\sigma^2(\Psi(x-x) - \Psi_S(x-x))$}. In order to justify this intuition, we need to show that $K_S$ is a covariance matrix, which follows from the fact that $\Psi_S$ is a positive definite function, as stated in the following lemma whose proof is given in Appendix \ref{App:pflemma}.
\begin{lemma}\label{lemma1}
{Suppose Assumption \ref{assumpsi} holds.} Then $\Psi_S$ is a positive definite function.
\end{lemma}
In order to study the asymptotic performance of KALE and KALEN, we consider a sequence of designs $X_m$, $m=1,2,\ldots$. We assume the following.
\begin{assum}\label{assumX}
% Let $\bar X_m=\{\bar x_1,...,\bar x_{l_m}\}$ denote the distinct design points corresponding to $X_m$.
% , and $n_1,...,n_m$ be the number of replicates on each distinct design point. Thus, $\sum_{j=1}^m n_j = n$. 
{The sequence of design points $X_m = \{x_1,\ldots,x_{n_m}\}$ satisfies that there exists a constant $C>0$ such that $h_{X_m}\leq Cq_{X_m}$ for all $m$, where $$q_{X_m}=\min_{x_j\neq x_k, x_j,x_k\in X_m}\|x_j - x_k\|_2/2,$$ {and $h_{X_m}$ is the fill distance of $X_m$ defined by \eqref{filldistance}}.}
\end{assum}
{\begin{remark}\label{rem:noreplicates}
Assumption \ref{assumX} implies that the \emph{distinct} design points are well separated. 
\end{remark}}
It is not hard to find designs satisfy this assumption. For example, grid designs satisfy Assumption \ref{assumX}. In the rest of paper we suppress the dependence of $X$ on $m$ for notational simplicity. It can be shown that if a Gaussian process has no intrinsic noise, then the MSPE of the corresponding best linear unbiased predictor converges to zero as the fill distance goes to zero {(see Lemma \ref{lemma2skzero} in Appendix \ref{app:A})}. Unlike a Gaussian process without input location error, we show that the limit of the MSPE of KALE and KALEN are usually not zero. In fact, \eqref{decomex} and Lemma \ref{lemma1} imply that the MSPE of KALEN is the MSPE of a Gaussian process with extrinsic error plus a non-zero constant. These results are stated in Theorem \ref{propMSPEKALEN}, whose proof is provided in Appendix \ref{App:pfThmkalen}. 

\begin{theorem}\label{propMSPEKALEN}
{Suppose Assumptions \ref{assumpsi} and \ref{assumX} hold.} The MSPE of KALEN \eqref{eq:MSPEKALEN} converges to $\sigma^2(\Psi(x-x) - \Psi_S(x-x))$ as the fill distance of the design points $h_X$ converges to zero, where $\Psi_S$ is defined in \eqref{eq:newKernel}.
\end{theorem}
In Theorem \ref{propMSPEKALEN}, we present a limit of the MSPE of KALEN. The limit $\sigma^2(\Psi(x-x) - \Psi_S(x-x))$ is usually not zero. This is expected for KALEN since there is a random error at the unobserved point $x$. The MSPE limit depends on two parts. One is the variance $\sigma^2$ and the other is the difference $\Psi(x-x) - \Psi_S(x-x)$. The variance $\sigma^2$ depends on the underlying process, while the difference depends on the probability density function of the noise $p(\cdot)$. Roughly speaking, the difference $\Psi(x-x) - \Psi_S(x-x)$ will be larger if the density $p(\cdot)$ is more spread out.

% One might expect that the MSPE of KALE converges to zero as the fill distance of the design points goes to zero. However, the following proposition shows that, in the case of Gaussian correlation functions and normally distributed intrinsic error, there is a positive lower bound on the MSPE of KALE. The proof can be found in Appendix \ref{PfofMSPELBKALE}.

% \begin{prop}\label{MSPELBKALE}
% Suppose $f$ is a Gaussian process with mean zero and covariance function $\sigma^2\Psi$. Suppose the correlation function $\Psi(s-t) = \exp(-\theta\|s-t\|_2^2)$ for some $\theta>0$, and the input noise $\epsilon_j \sim N(0,\sigma_\epsilon^2I_d)$ are i.i.d., where $N(0,\sigma_\epsilon^2I_d)$ is a mean zero normal distribution with covariance matrix $\sigma_\epsilon^2I_d$. Then for any design $X = \{x_1,\ldots,x_n\} \subset \Omega$, the MSPE of KALE, defined in \eqref{eq:MSPEKALE}, has a lower bound
% \begin{align}\label{eqkalegalb}
%     \sigma^2\bigg(1- \frac{(1+4\sigma^2_\epsilon\theta)^{d/2}}{(1+2\sigma^2_\epsilon\theta)^d}\bigg).
% \end{align}
% \end{prop}

% \begin{remark}
% {From Proposition \ref{MSPELBKALE}, it can be seen that the lower bound is large if $d$ is large. This implies that the prediction error is large when the dimension is high.}
% \end{remark}

\section{Comparison Between KALE/KALEN and Stochastic Kriging}\label{sec:CompareNugg}

It is argued in \cite{cressie2003spatial} and \cite{stein1999interpolation} that using a nugget parameter is one way to counteract the influence of noise within the inputs. Therefore, it is natural to ask whether stochastic Kriging {(i.e., Kriging with a nugget parameter)} is a good approximation method to predict the value at an unobserved point, since it is not the best linear unbiased predictor under the settings of Gaussian process with input location error. In this paper, we show that the MSPE of stochastic Kriging has the same limit as the MSPE of KALEN, and provide an upper bound on the MSPE of stochastic Kriging if the unobserved point has no noise, as stated in Theorem \ref{Thm:main}. The proof can be found in Appendix \ref{App:pfThmmain}. 
\begin{theorem}\label{Thm:main}
{Suppose Assumptions \ref{assumpsi} and \ref{assumX} hold.} Let $\mu>0$ be {any fixed} constant. A stochastic Kriging predictor of a Gaussian process with input location error is defined as
\begin{eqnarray}\label{eq:BLUPSK1}
\hat f_S(x) = \Psi(x-X)(\Psi(X-X)+\mu I_n)^{-1}Y,
\end{eqnarray}
where $\Psi(x-X) = (\Psi(x-x_1),\ldots,\Psi(x-x_n))^T$ and $\Psi(X-X) = (\Psi(x_j-x_k))_{jk}$. 

(i) Suppose there is noise at an unobserved point. 
The MSPE of the predictor \eqref{eq:BLUPSK1} has the same limit as KALEN, which is $\sigma^2(\Psi(x-x) - \Psi_S(x-x))$, where $\Psi_S$ is as defined in (\ref{eq:newKernel}), when the fill distance of $X$ goes to zero.

(ii) Suppose there is no noise at an unobserved point. An asymptotic upper bound on the MSPE of the predictor \eqref{eq:BLUPSK1} is
\begin{align}\label{eq:uppBKALE}
\frac{1.04\sigma^2}{(2\pi)^{d/2}}\int_{\mathbb{R}^d} \big|1-|b(t)|\big|^2\mathcal{F}(\Psi)(t)dt,
\end{align}
where $\mathcal{F}(\Psi)$ is the Fourier transform of $\Psi$ and $b(t) = \mathbb{E}(e^{i\epsilon^Tt})$ is the characteristic function of $p(\cdot)$. 
\end{theorem}

\begin{remark}
We say $b$ is an asymptotic upper bound on a sequence $a_n$, if there exists a sequence $b_n$ such that $a_n\leq b_n$ for all $n=1,2,\ldots$, and $\lim_{n\rightarrow \infty} b_n= b$.
\end{remark}

{\begin{remark}
The constant 1.04 in \eqref{eq:uppBKALE} is not essential. It can be changed to any constant greater than one, but a larger constant leads to a ``slower'' convergence speed.
\end{remark}

\begin{remark}
Note that KALE is the best linear unbiased predictor when an unobserved point has no noise. Therefore, the upper bound of MSPE for stochastic Kriging is also an upper bound of MSPE for KALE. For an illustration of the upper bound and lower bound of the MSPE of KALE, see Example \ref{explgauss}.
\end{remark}
}

Theorem \ref{Thm:main} shows that the predictor \eqref{eq:BLUPSK1} is as good as KALEN asymptotically. The following proposition states that if the noise is small, then \eqref{eq:uppBKALE} can be controlled. The proof of Proposition \ref{Upprop} can be found in Appendix \ref{pfUpprop}.

\begin{prop}\label{Upprop}
Suppose Assumption \ref{assumpsi} holds, and $\{\epsilon_n\}$ is a sequence of independent random vectors that converges to $0$ in distribution. Let 
\begin{align}
a_n = \frac{\sigma^2}{(2\pi)^{d/2}}\int_{\mathbb{R}^d} \big|1-|b_n(t)|\big|^2\mathcal{F}(\Psi)(t)dt,
\end{align}
where $b_n(t) = \mathbb{E}(e^{i\epsilon_n^Tt})$. Then $a_n$ converges to zero.
\end{prop}

{\begin{expl}\label{explgauss}
Consider a Gaussian process $f$ with mean zero and covariance function $\sigma^2\Psi$. Suppose the correlation function $\Psi(s-t) = \exp(-\theta\|s-t\|_2^2)$ with $\theta>0$, and the intrinsic noise $\epsilon_j \sim N(0,\sigma_\epsilon^2I_d)$ are i.i.d., where $N(0,\sigma_\epsilon^2I_d)$ is a mean zero normal distribution with covariance matrix $\sigma_\epsilon^2I_d$. By Theorem \ref{Thm:main}, the limit of the MSPE of KALEN and stochastic Kriging is $\sigma^2(\Psi(x-x) - \Psi_S(x-x))$, which can be computed by
\begin{align}\label{egkale}
    \sigma^2(\Psi(x-x) - \Psi_S(x-x)) = & \sigma^2\bigg(1- \iint\Psi(x+\epsilon_1-(x+\epsilon_2))p(\epsilon_1)p(\epsilon_2)d\epsilon_1 d\epsilon_2\bigg)\nonumber\\
    = & \sigma^2 - r_N(x,x) =  \sigma^2 - r_N(x_j,x_j) =\sigma^2\bigg(  \frac{(1+4\sigma^2_\epsilon\theta)^{d/2} - 1}{(1+4\sigma^2_\epsilon\theta)^{d/2}}\bigg),
\end{align}
where $r_N(x_j,x_j)$ is as in \eqref{eq:KernelMnormal} with $x=x_j$.

If there is no noise at an unobserved point $x$, Theorem \ref{Thm:main} states that an asymptotic upper bound of MSPE for stochastic Kriging is
\begin{align*}
\frac{1.04\sigma^2}{(2\pi)^{d/2}}\int_{\mathbb{R}^d} \big|1-|b(t)|\big|^2\mathcal{F}(\Psi)(t)dt. 
\end{align*}
Note that the characteristic function of $N(0,\sigma_\epsilon^2I_d)$ is $b(t) = \mathbb{E}(e^{i\epsilon^Tt}) = e^{-\frac{1}{2}\sigma_\epsilon^2t^Tt}$, and $\mathcal{F}(\Psi)(t) = \theta^{-d/2}e^{-\frac{t^Tt}{4\theta}}$. Thus, the upper bound can be computed by
\begin{align}\label{egkalen}
\frac{1.04\sigma^2}{(2\pi)^{d/2}}\int_{\mathbb{R}^d} \big|1-|b(t)|\big|^2\mathcal{F}(\Psi)(t)dt & = \frac{1.04\sigma^2}{(2\pi\theta)^{d/2}}\int_{\mathbb{R}^d} (1 - e^{-\sigma_\epsilon^2t^Tt/2})^2e^{-\frac{t^Tt}{4\theta}}dt\nonumber\\
& = 
1.04\sigma^2\left(1+\frac{1}{(1+4\sigma_\epsilon^2\theta)^{d/2}}- \frac{2}{(1+2\sigma_\epsilon^2\theta)^{d/2 }}\right).
\end{align}

% Proposition \ref{MSPELBKALE} presents a lower bound of the MSPE of KALE, which is $$\sigma^2\bigg(1- \frac{(1+4\sigma^2_\epsilon\theta)^{d/2}}{(1+2\sigma^2_\epsilon\theta)^d}\bigg).$$ Comparing this lower bound with \eqref{egkalen}, we can see that the difference is
% \begin{align}\label{egkalediff}
%     & 1.04\sigma^2\left(1+\frac{1}{(1+4\sigma_\epsilon^2\theta)^{d/2}}- \frac{2}{(1+2\sigma_\epsilon^2\theta)^{d/2 }}\right) - \sigma^2\bigg(1- \frac{(1+4\sigma^2_\epsilon\theta)^{d/2}}{(1+2\sigma^2_\epsilon\theta)^d}\bigg)\nonumber\\
%     = & 0.04\sigma^2\left(1+\frac{1}{(1+4\sigma_\epsilon^2\theta)^{d/2}}- \frac{2}{(1+2\sigma_\epsilon^2\theta)^{d/2 }}\right) + \sigma^2\left(\frac{(1+4\sigma^2_\epsilon\theta)^{d/2}}{(1+2\sigma^2_\epsilon\theta)^d}+\frac{1}{(1+4\sigma_\epsilon^2\theta)^{d/2}}- \frac{2}{(1+2\sigma_\epsilon^2\theta)^{d/2 }}\right)\nonumber\\
%     = & 0.04\sigma^2\left(1+\frac{1}{(1+4\sigma_\epsilon^2\theta)^{d/2}}- \frac{2}{(1+2\sigma_\epsilon^2\theta)^{d/2 }}\right) + \sigma^2\frac{\big((1+4\sigma_\epsilon^2\theta)^{d/2} - (1+2\sigma_\epsilon^2\theta)^{d/2}\big)^2}{(1+2\sigma_\epsilon^2\theta)^d(1+4\sigma_\epsilon^2\theta)^{d/2}}.
% \end{align}

Figure \ref{figbound} shows the plot of limit \eqref{egkale} and the asymptotic upper bound \eqref{egkalen}
% , and the difference \eqref{egkalediff} between the upper bound \eqref{egkalen} and the lower bound \eqref{eqkalegalb} 
with $\theta =1$ and $\sigma^2 =1$. It can be seen that as the variance of noise increases, both \eqref{egkale} and \eqref{egkalen} increase, and \eqref{egkalen} is larger than \eqref{egkale}. From Panel 1 and Panel 2 of Figure \ref{figbound}, the error is larger if the dimension of the space is larger. This indicates that as in many statistic problems, Gaussian process with input location error is also influenced by the dimension.
% The differences between the upper bound \eqref{egkalen} and the lower bound \eqref{eqkalegalb} is small when dimension is not high, but is larger when dimension becomes larger. 

\begin{figure}[h!]
    \centering
    \begin{subfigure}
        \centering
        \includegraphics[height=2.5in]{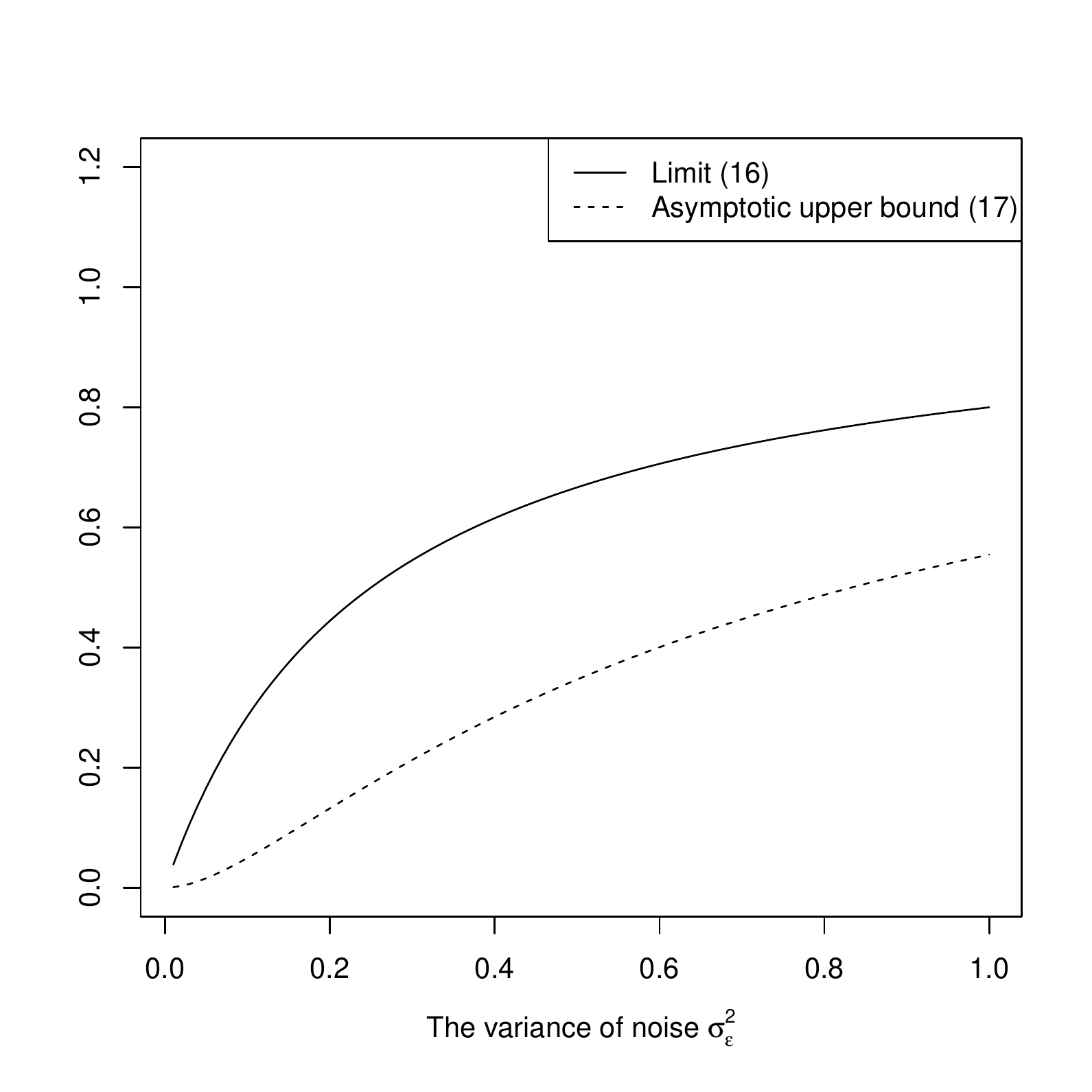}
        \end{subfigure}
    \begin{subfigure}
        \centering\includegraphics[height=2.5in]{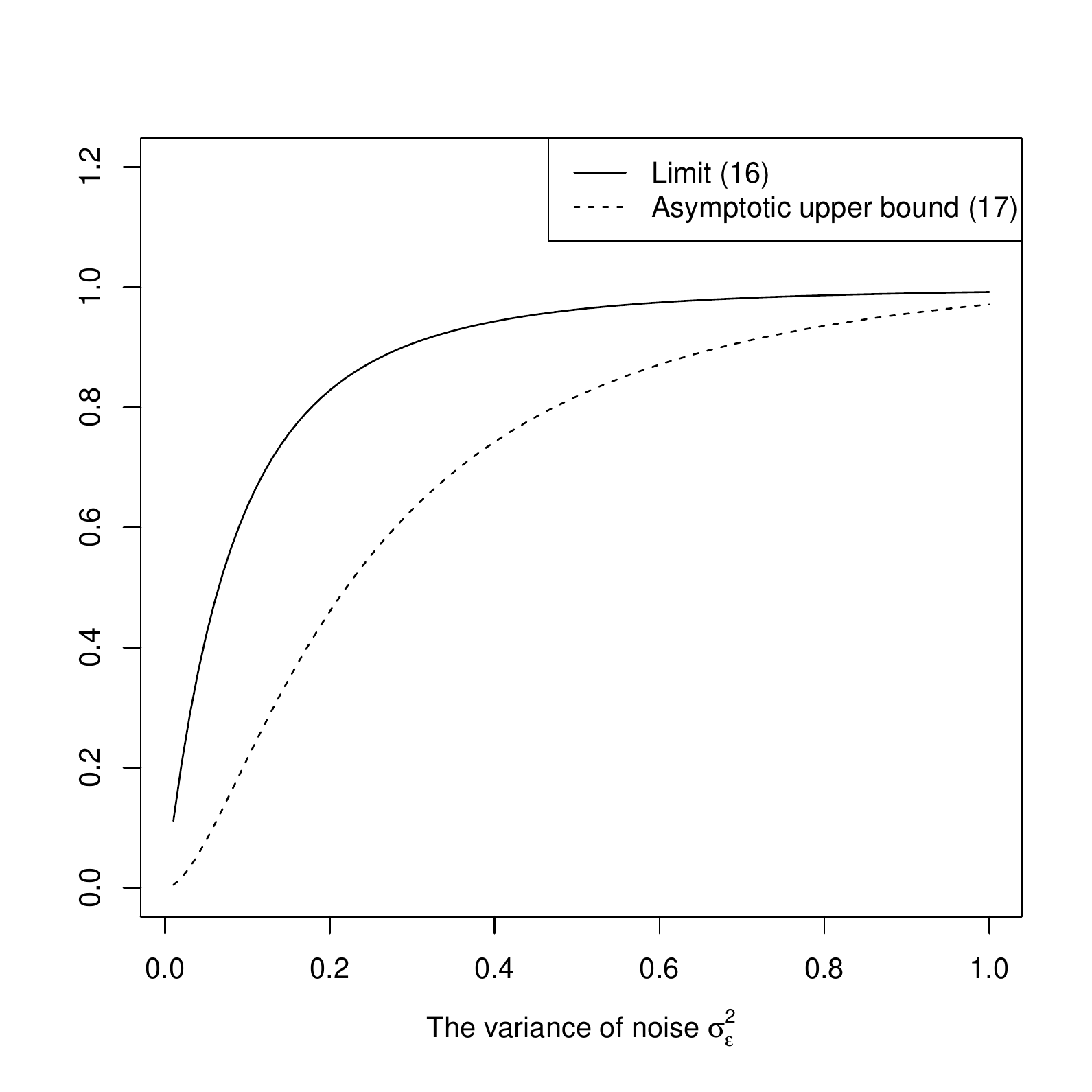}\end{subfigure}%\\
    % \subfloat{\includegraphics[height=2.5in]{bound2dX}}
    % \subfloat{\includegraphics[height=2.5in]{bound6dX}}
    \caption{The limit \eqref{egkale} and the asymptotic upper bound \eqref{egkalen}
    % , and the difference \eqref{egkalediff} between the upper bound \eqref{egkalen} and the lower bound \eqref{eqkalegalb} 
    with $\theta =1$ and $\sigma^2 =1$. \textbf{Panel 1:} $d = 2$. \textbf{Panel 2:} $d = 6$. }\label{figbound} %\textbf{Panel 3:} Relative difference with $d = 2$. \textbf{Panel 4:} Relative difference with $d = 6$.
\end{figure}
\end{expl}
}

One advantage of stochastic Kriging is that we can simplify the calculation since we do not need to calculate the integrals in \eqref{eq:RvectorNoi}, (\ref{eq:KernelMatrixNoi}), and (\ref{eq:RvectorNoiX}). If the noise is small and the fill distance is small, Theorem \ref{Thm:main} and Proposition \ref{Upprop} state that the MSPE of the stochastic Kriging predictor \eqref{eq:BLUPSK1} can be comparable with the best linear unbiased predictor. 

It is argued in \cite{cervone2015gaussian} that since the integrated covariance function in (\ref{eq:KernelMatrixNoi}) is not the same as the covariance function in the original Gaussian process without location error, a nugget parameter alone cannot capture the effect of location error. {While it is }true that the MSPE of KALE or KALEN is the smallest among all the linear unbiased predictors, our results also show that with {any fixed constant} nugget parameter, the predictor \eqref{eq:BLUPSK1} is as good as KALEN asymptotically, and there is little difference between KALE and the predictor \eqref{eq:BLUPSK1} if the variance of the intrinsic noise and the fill distance are small. {If the sample size $n$ is large, the computational cost of KALE/KALEN and stochastic Kriging will be high, because the computation of a dense matrix inverse is $O(n^3)$. Note that the dense matrix inverse also appears in ordinary Gaussian process modeling. If the sample size is small and the variance of the intrinsic noise is large, as suggested by numeric studies, the difference between the MSPE of KALE or KALEN and stochastic Kriging is large, thus stochastic Kriging with a single nugget parameter may not lead to a good predictor in this case.}

\section{Parameter Estimation}\label{sec:ParaEst}
{Let $\Psi_{\theta^{(1)}}$ be a class of correlation functions and $p_{\theta^{(2)}}(\cdot)$ be a class of probability density functions indexed by $(\theta^{(1)}, \theta^{(2)})\in \Theta$, respectively, where $\Theta$ is a parameter space.} An intuitive approach to estimate the parameters is maximum likelihood estimation. Up to a multiplicative constant, the likelihood function is
\begin{align}\label{eq:likelihoodBig}
\ell(\sigma^2,\theta^{(1)}, \theta^{(2)};X,Y) \propto \int\ldots\int {\rm det}(\Sigma_1)^{-1/2} e^{-\frac{1}{2}Y^T\Sigma_1 ^{-1}Y}
p_{\theta^{(2)}}(\epsilon_1)\ldots p_{\theta^{(2)}}(\epsilon_n)d\epsilon_1\ldots d\epsilon_n,
\end{align}
where $\Sigma_1 = (\sigma^2\Psi_{\theta^{(1)}}(x_j+\epsilon_j-(x_k+\epsilon_k)))_{jk}$, and ${\rm det}(A)$ is the determinant of a matrix $A$. Unfortunately, the integral in (\ref{eq:likelihoodBig}) is difficult to calculate, because the dimension of the integral increases as the sample size increases. In this work, we use a pseudo-likelihood approach proposed by \cite{cressie2003spatial}. Define
\begin{align}\label{eq:likelihoodpseudo}
\ell_g(\sigma^2,\theta^{(1)}, \theta^{(2)};X,Y) = (2\pi)^{-n/2} {\rm det}(K_{(\theta^{(1)},\theta^{(2)})})^{-1/2} \exp\bigg(-\frac{1}{2}Y^T K_{(\theta^{(1)},\theta^{(2)})}^{-1}Y\bigg),
\end{align}
where $\sigma^2,\theta^{(1)}, \theta^{(2)}$ are parameters we want to estimate, and $K_{(\theta^{(1)},\theta^{(2)})}$ is defined in \eqref{eq:KernelMatrixNoi} {{by replacing $\Psi$ and $p(\cdot)$ with $\Psi_{\theta^{(1)}}$ and $p_{\theta^{(2)}}(\cdot)$, respectively.} The maximum pseudo-likelihood estimator can be defined as
\begin{align}\label{pseudomleestimator}
(\hat\sigma_1^2,\hat \theta_1^{(1)}, \hat \theta_1^{(2)}) = \argsup_{(\sigma^2,\theta^{(1)}, \theta^{(2)})} \ell_g(\sigma^2,\theta^{(1)}, \theta^{(2)};X,Y).
\end{align}
{If \eqref{pseudomleestimator} has multiple solutions, we choose any one from them. Because of non-identifiability, parameters inside the Gaussian process $(\sigma^2,\theta^{(1)})$ and parameters inside the probability density function of input variable noise $\theta^{(2)}$ cannot be estimated simultaneously \cite{cervone2015gaussian}. }
%However, the quality of correlation functions based prediction tools only depends weakly on the accuracy of the parameter estimation. For most correlation functions, for example, the isotropic Mat\'ern kernel correlation functions and the isotropic Gaussian correlation functions, the spaces of functions do not depend on the variance parameter $\sigma$. 
The properties of the pseudo-likelihood approach are discussed in \cite{cervone2015gaussian}. Here we list a few of them. First, the pseudo-score provides an unbiased estimation equation, i.e.,
$$\mathbb{E}( S (\sigma^2,\theta^{(1)}, \theta^{(2)};X,Y)) = \mathbb{E}(\nabla\log(\ell_g(\sigma^2,\theta^{(1)}, \theta^{(2)};X,Y)))=0.$$ Second, the covariance matrix of the pseudo-score $\mathbb{E}( S (\sigma^2,\theta^{(1)}, \theta^{(2)};X,Y) S (\sigma^2,\theta^{(1)}, \theta^{(2)};X,Y)^T)$ and the expected negative Hessian of the log pseudo-likelihood $\mathbb{E}\big(\frac{\partial^2}{\partial \theta_j\partial \theta_k}\log(\ell_g(\sigma^2,\theta^{(1)}, \theta^{(2)};X,Y))\big)$ can be calculated, where $\theta_j$ and $\theta_k$ are elements in $(\sigma^2,\theta^{(1)}, \theta^{(2)})$. However, the consistency of parameters estimated by pseudo-likelihood in the case of Gaussian process has not been theoretically justified to the best of our knowledge.

If we use stochastic Kriging, the corresponding (misspecified) log likelihood function is, up to an additive constant, 
\begin{align}\label{eq:likeliwithnug}
\ell_{nug}(\sigma^2,\theta^{(1)}, \mu;X,Y) = -\frac{1}{2}\log ({\rm det}(\Psi_{\theta^{(1)}}(X-X)+\mu I_n)) -\frac{1}{2}Y^T (\Psi_{\theta^{(1)}}(X-X)+\mu I_n)^{-1}Y.
\end{align}
The maximum likelihood estimator of $(\sigma^2,\theta^{(1)}, \mu)$ is defined by
\begin{eqnarray}\label{MLE}
(\hat \sigma_2^2,\hat \theta_2^{(1)}, \hat \mu)=\argsup_{(\sigma^2,\theta^{(1)}, \mu)} \ell_{nug}(\sigma^2,\theta^{(1)}, \mu;X,Y).
\end{eqnarray}

Note that \eqref{eq:likeliwithnug} is the log likelihood function for a Gaussian process with only extrinsic noise. Thus it is misspecified, and the estimated parameters may also be misspecified. However, it has been shown by the well-known works \cite{ying1991asymptotic} and \cite{zhang2004inconsistent} that the Gaussian process model parameters in the covariance functions may not have consistent estimators. Therefore, using Gaussian process models for prediction may be more meaningful than for parameter estimation. In fact, the parameter estimates do not significantly influence our theoretical results on the MSPE of KALE, KALEN and stochastic Kriging, in the sense of the following theorem, {whose proof is presented in Appendix \ref{pfthmparaest}.}
\begin{theorem}\label{thmparaest}
Suppose for some constant $C>0$, $1/C \leq \hat{\mu} \leq C$ holds for all $n$. Let $\hat \Psi_1$ and $\hat \Psi_2$ be the correlation functions with estimated parameters $\hat \theta_1^{(1)}$ and $\hat \theta_2^{(1)}$ as in \eqref{pseudomleestimator} and \eqref{MLE}, respectively. Let $\hat p(\cdot)$ be the probability density function with estimated parameters $\hat \theta_1^{(2)}$. Let $\hat \Psi_S$ be as in \eqref{eq:newKernel} with estimated parameters. Potential dependency of $\hat{\mu}$, $\hat \Psi_1$, $\hat \Psi_2$, $\hat p(\cdot)$, and $\hat \Psi_S$ on $n$ is suppressed for notational simplicity. Assume the following.

(1) There exists a constant $A_1$ such that for all $n$
\begin{align}
    \max\left\{\left\|\frac{\mathcal{F}(\Psi)}{\mathcal{F}(\hat\Psi_S)}\right\|_{L_\infty},\left\|\frac{\mathcal{F}(\Psi)}{\mathcal{F}(\hat\Psi_1)}\right\|_{L_\infty},\left\|\frac{\mathcal{F}(\Psi)}{\mathcal{F}(\hat\Psi_2)}\right\|_{L_\infty}\right\} \leq A_1.
\end{align}

(2) {For some fixed Sobolev space, Assumption \ref{assumpsi} holds for all $\hat \Psi_1$ and $\hat \Psi_2$, and the embedding constants have a uniform upper bound for all $n$.}

(3) Assumption \ref{assumX} holds for the sequence of designs $X$.

(4) All probability density functions $\hat p(\cdot)$ are continuous, have mean zero and second moment. The second moments of all $\hat p(\cdot)$ have a uniform positive lower bound and upper bound for all $n$.

Then the following statements are true.

(i) Suppose there is noise at an unobserved point $x$.  Then the MSPE of KALEN and stochastic Kriging have the limit $\sigma^2(\Psi(x-x) - \Psi_{S}(x-x))$ when the fill distance of $X$ goes to zero, where $\Psi_S$ is defined in \eqref{eq:newKernel}.

(ii) Suppose there is no noise at an unobserved point $x$. An asymptotic upper bound on the MSPE of stochastic Kriging is
\begin{align*}
\frac{1.04\sigma^2}{(2\pi)^{d/2}}\int_{\mathbb{R}^d} |1-|b(t)|^2|\mathcal{F}(\Psi)(t)dt,
\end{align*}
where $b(t) = \mathbb{E}(e^{i\epsilon^Tt})$ is the characteristic function of $h$. Furthermore, if $\hat p(\cdot)=p(\cdot)$ and $\left\|\frac{\mathcal{F}(\hat\Psi_1)}{\mathcal{F}(\Psi)}\right\|_{L_\infty}\leq A_2$, an asymptotic upper bound on the MSPE of KALE is
\begin{align*}
\frac{1.04A_1A_2\sigma^2}{(2\pi)^{d/2}}\int_{\mathbb{R}^d} |1-|b(t)|^2|\mathcal{F}(\Psi)(t)dt.
\end{align*}
\end{theorem}
Theorem \ref{thmparaest} states that if the pseudo-likelihood $ \ell_g$ and the misspecified log likelihood $\ell_{nug}$ can provide reasonable estimated parameters, then we have the following: (1) If an unobserved point has noise, the limit of the MSPE of KALEN and stochastic Kriging remains the same; and (2) If an unobserved point has no noise, the upper bounds on the MSPE of KALE and stochastic Kriging can be obtained. The limit and upper bounds are small if the noise is small. The upper bound for the MSPE of stochastic kriging is the same as the bound in Theorem \ref{Thm:main}. However, the upper bound for the MSPE of KALE is inflated by $A_1A_2$. We believe this inflation is not necessary and can be improved. In sum, the parameter estimation does not significantly influence our theoretical analysis.

The computation complexity of \eqref{MLE} is about the same as that of \eqref{pseudomleestimator}, if \eqref{eq:KernelMatrixNoi} can be calculated analytically. Unfortunately, \eqref{eq:KernelMatrixNoi} usually does not have a closed form, which substantially increases the computation time of solving \eqref{pseudomleestimator}.

\section{Numeric Results}\label{secNumerical}
In this section, we report some simulation studies to investigate the numeric performance of KALE, KALEN and stochastic Kriging. In Example 1, we use Gaussian correlation functions to fit a 1-d function, where the predictor \eqref{BLUPNoN} has analytic form. In Example 2, we use Mat\'ern correlation functions to fit a 2-d function, where the integrals in \eqref{eq:RvectorNoi} and \eqref{eq:KernelMatrixNoi} need to be calculated by Monte-Carlo sampling \cite{cressie2003spatial}.

\subsection{Example 1}\label{Eg1subsec}
Suppose the underlying function is $f(x)=\sin(2\pi x/10)+0.2\sin(2\pi x/2.5)$, $x\in [0,8]$ \cite{higdon2002space}. The design points are selected to be 161 evenly spaced points on $[0, 8]$. The intrinsic noise is chosen to be mean zero normally distributed with the variances $0.05k$, for $k=1,2,3,4$. We use a Gaussian correlation function $\Psi(s-t) = \sigma^2\exp(-\theta\|s - t\|_2^2)$ to make predictions, and use the pseudo-likelihood approach presented in Section \ref{sec:ParaEst} to estimate the unknown parameters $\sigma^2, \theta$ and the variance of noise $\sigma_\epsilon^2$. For each variance of intrinsic noise, we approximate the squared $L_2$ error $\|f - \hat f\|_2^2$ by $\frac{8}{n}\sum_{i=1}^n(f(x_i) - \hat f(x_i))^2$, where the $x_i$'s are 8001 evenly spaced points on $[0,8]$. Then we run 100 simulations and take the average of $\frac{8}{n}\sum_{i=1}^n(f(x_i) - \hat f(x_i))^2$ to estimate $\mathbb{E}\|f - \hat f\|_2^2$. We estimate $\mathbb{E}\|y - \hat y\|_2^2$ by a similar approach, {i.e., estimate $\mathbb{E}\|y - \hat y\|_2^2$ by the average of $\frac{8}{n}\sum_{i=1}^n(y(x_i) - \hat y(x_i))^2$ of 100 simulations, where $y(x_i) = f(x_i + \epsilon_i)$ and $\epsilon_i$'s are intrinsic noise.} 
% The ``true value'' of $y(x)$ is approximated by $\bar y(x_i)=\frac{1}{1000}\sum_{k=1}^{1000}f(x_i+\epsilon_k)$ with $\epsilon_k$'s i.i.d. random variables with distribution $N(0,0.05k)$.} 
Recall that $\mathbb{E}\|f - \hat f\|_2^2$ and $\mathbb{E}\|y - \hat y\|_2^2$ are related to KALE and KALEN, respectively. With an abuse of terminology, we still call $\mathbb{E}\|f - \hat f\|_2^2$ and $\mathbb{E}\|y - \hat y\|_2^2$ MSPE.

% In order to make a comprehensive comparison, we also include the results from the MCMC method. After 3000 burn-in runs, we run 40 iterations for prediction, i.e., calculating $\hat f_B(x)$ and $\hat y_B(x)$ in \eqref{mcmcpredictor}. The prior we choose is $\theta \sim {\rm Unif}[0.1,0.2]$, $\sigma \sim {\rm Unif}[0,1]$, and {$\sigma_\epsilon \sim {\rm Unif}[0,1.2\sqrt{0.05k}]$.}

The RMSPEs, which are the square roots of MSPEs, for KALE/KALEN and stochastic Kriging, are shown in Table \ref{tab:eg1dcasenonoise}/Table \ref{tab:eg1dcasenoise}, respectively.

\begin{table}[h]
\centering
\begin{tabular}{|c|c|c|c|}
\hline  
$\sigma_\epsilon^2$ & RMSPE of KALE & RMSPE of stochastic   & Difference \\
           &              &  Kriging &  \\
\hline
0.05 & 0.1147 & 0.1209 & 0.0062 \\
0.10 & 0.1528 & 0.1764 & 0.0236 \\
0.15 & 0.1917 & 0.2364 & 0.0448 \\
0.20 & 0.2380 & 0.3149 & 0.0769\\
\hline
\end{tabular}
\caption{Comparison of the RMSPE for KALE and stochatic Kriging: 1-d function with Gaussian correlation function. In fourth column, difference $=$ 3rd column $-$ 2nd column, i.e., the RMSPE of stochastic Kriging $-$ the RMSPE of KALE.}
\label{tab:eg1dcasenonoise}
\end{table}

\begin{table}[h]
\centering
\begin{tabular}{|c|c|c|c|c|}
\hline  
$\sigma_\epsilon^2$  & RMSPE of KALEN  & RMSPE of stochastic & Difference\\
           &                &  Kriging & \\
\hline
0.05 & 0.3627 & 0.3619 &  $-0.0014$ \\
0.10 & 0.4940 & 0.4931 &  $-0.0009$ \\
0.15 & 0.5884 & 0.5885 &  $0.0001$ \\
0.20 & 0.6651 & 0.6704 &  $0.0053$ \\
\hline
\end{tabular}
\caption{Comparison of the RMSPE for KALEN and stochatic Kriging: 1-d function with Gaussian correlation function. In fourth column, difference $=$ 3rd column $-$ 2nd column, i.e., the RMSPE of stochastic Kriging $-$ the RMSPE of KALEN.}
\label{tab:eg1dcasenoise}
\end{table}

It can be seen from Tables \ref{tab:eg1dcasenonoise} and \ref{tab:eg1dcasenoise} that the RMSPE of KALE/KALEN and stochastic Kriging decreases as the variance of the intrinsic noise decreases. This corroborates the results in Theorem \ref{Thm:main} and Proposition \ref{Upprop}. The difference of RMSPE between KALE/KALEN and stochastic Kriging also decreases when the variance of the intrinsic noise decreases. Comparing Table \ref{tab:eg1dcasenoise}  with Table \ref{tab:eg1dcasenonoise}, it can be seen that the RMSPE of KALEN is larger than that of KALE. This is reasonable because KALEN predicts $y(x)$, which includes an error term while $f(x)$ does not. The computation of KALE/KALEN has the same complexity as the stochastic Kriging in this example, because a Gaussian correlation function is used, and the integrals in (\ref{eq:KernelMatrixNoi}) and (\ref{eq:RvectorNoiX}) can be calculated analytically. 
% {In Table \ref{tab:eg1dcasenoise}, the difference of RMSPE between KALEN and stochastic Kriging is negative when the variance of noise is small. This may be because the solution to the maximum likelihood estimation of KALEN does not achieve the global optima, due to the difficulty of optimization problems. }

In order to further understand the performance of KALE/KALEN and stochastic Kriging, {two realizations among the 100 simulations for Table \ref{tab:eg1dcasenonoise} and Table \ref{tab:eg1dcasenoise} are illustrated in Panel 1 and Panel 2 of Figure \ref{1dCaseFig}, respectively,} where the variance of the intrinsic noise is chosen to be 0.05. In Panel 1 of Figure \ref{1dCaseFig}, the circles are the collected data points. The true function, the prediction curves of KALE and stochastic Kriging are denoted by solid line, dashed line and dotted line, respectively. It can be seen from the figure that both KALE and stochastic Kriging approximate the true function well. {In Panel 2 of Figure \ref{1dCaseFig}, the dots are the samples of $y(x)$ on 8001 testing points. It can be seen that the samples are around the predictions of KALEN and stochastic Kriging, but with much more fluctuations. This shows that the RMSPE in Table \ref{tab:eg1dcasenoise} is larger than those in Table \ref{tab:eg1dcasenonoise}.}

\begin{figure}[h!]
    \begin{subfigure}
        \centering\includegraphics[height=3in]{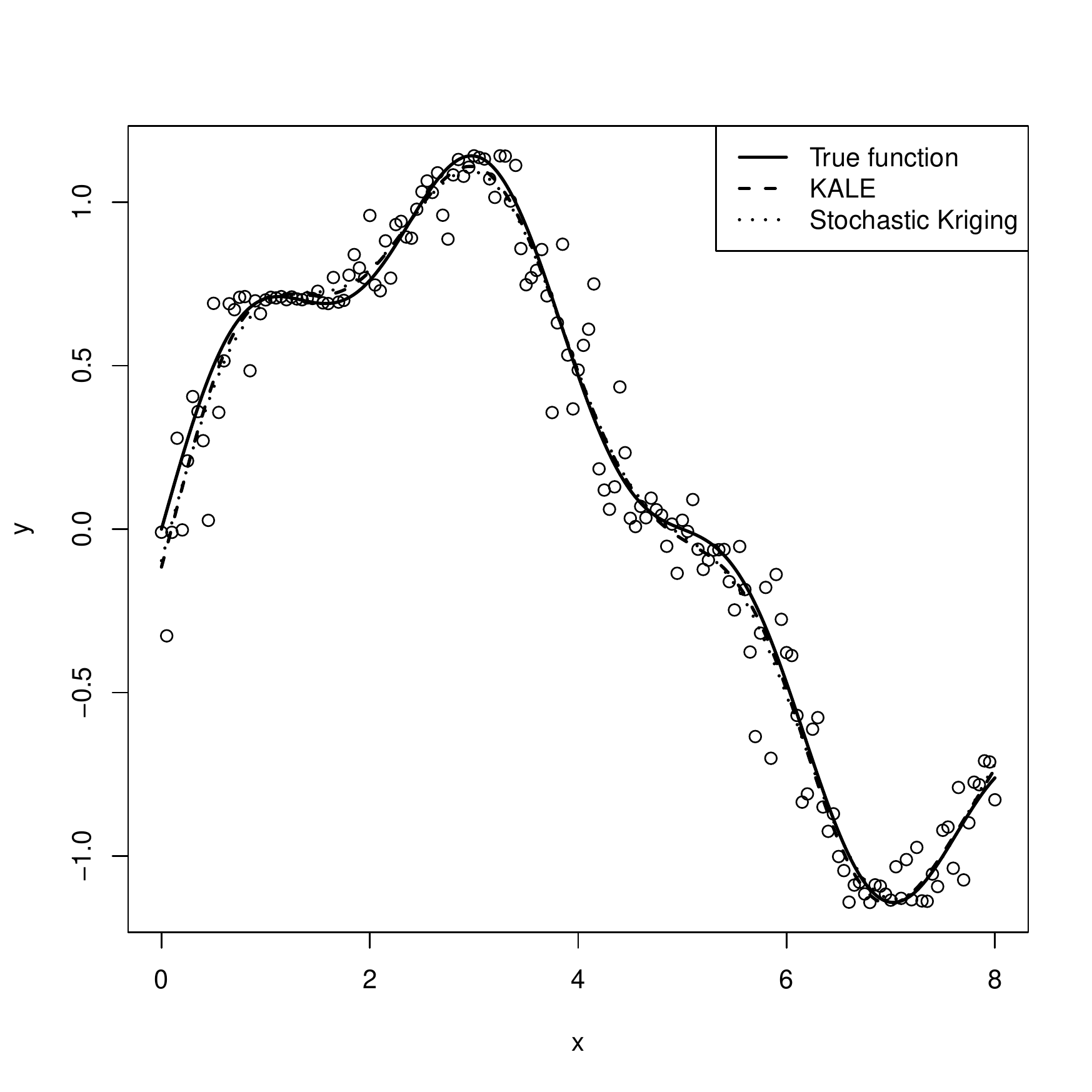}\end{subfigure}
   \begin{subfigure}
        \centering\includegraphics[height=3in]{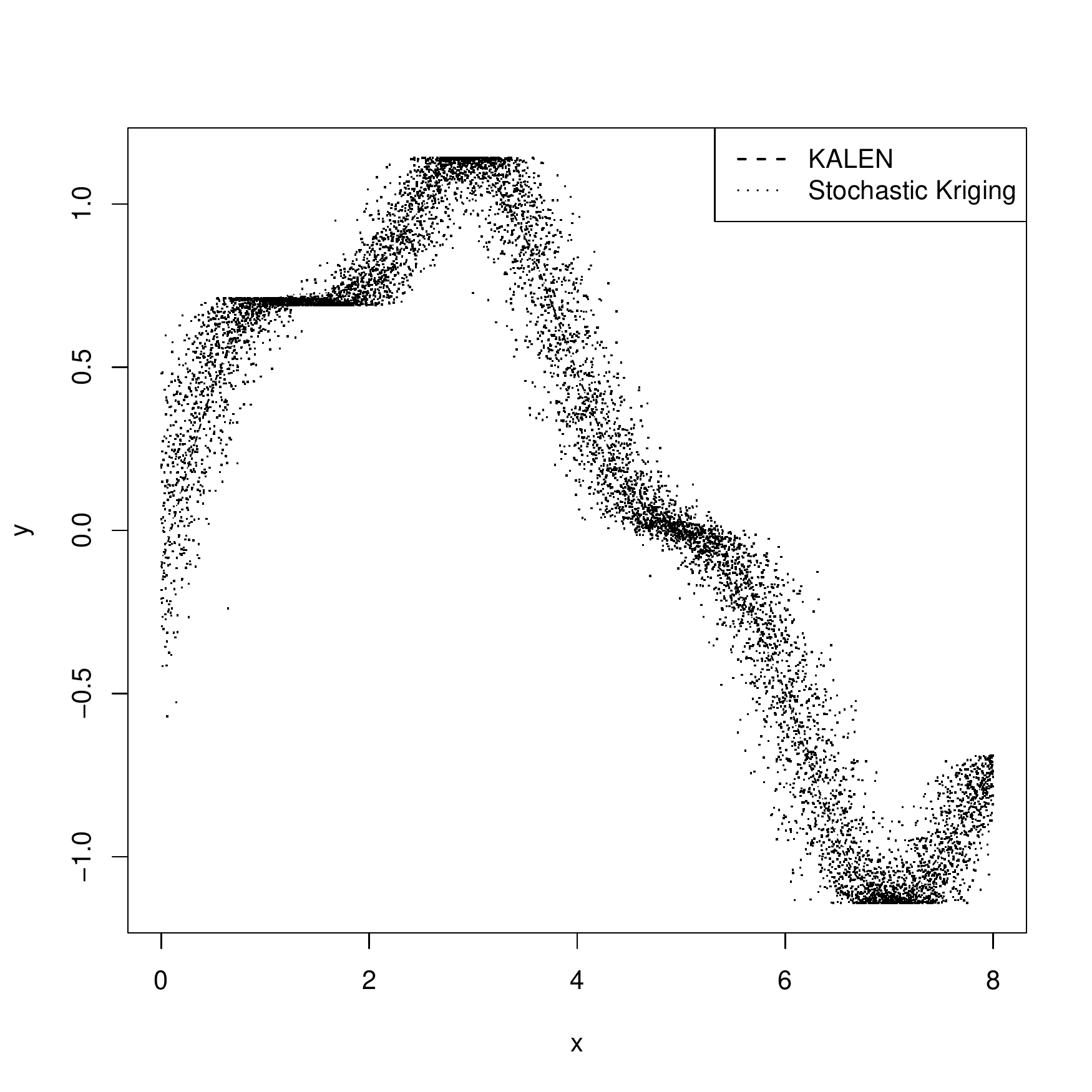}\end{subfigure}
    \caption{\textbf{Panel 1:} An illustration of KALE and stochastic Kriging. The true function, the prediction curves of KALE and stochastic Kriging are denoted by solid line, dashed line and dotted line, respectively. The circles are the observed data points. \textbf{Panel 2:} An illustration of KALEN and stochastic Kriging. The dots are the samples of $y(x)$ on testing points.The prediction curves of KALEN and stochastic Kriging are denoted by dashed line and dotted line. }\label{1dCaseFig}
\end{figure}

{
We also include the confidence interval results in this subsection. It is known \cite{cervone2015gaussian} that there is no nontrivial structure for $\epsilon$ (that is, $\epsilon \not\equiv 0$) such that $y$ is a Gaussian process on $\Omega$. Since there is no closed form for the distribution of $\hat f(x)$ (or $\hat y(x)$), we use Gaussian approximation. Specifically, we treat $\hat f(x)$ (or $\hat y(x)$) as normally distributed and compute the pointwise conditional variance $\hat \sigma_f(x)^2$ (or $\hat \sigma_y(x)^2$). Then we compute the pointwise confidence interval of Gaussian process, defined by $[\hat f(x) - q_\beta \hat \sigma_f(x), \hat f(x) + q_\beta \hat \sigma_f(x)]$ (or $[\hat y(x) - q_\beta \hat \sigma_y(x), \hat y(x) + q_\beta \hat \sigma_y(x)]$)  with confidence level $(1 - \beta)100\%$, where $q_\beta$ denote the $(1 - \beta/2)$th quantile of standard normal distribution. We select $\beta = 0.05$ and use coverage rate to quantify the quality of the confidence interval, where the coverage rate is the proportion of the time that the interval contains the true value. However, the length of confidence interval of stochastic Kriging for Gaussian process with only extrinsic error converges to zero, which does not reflect the fact that the actual MSPE of stochastic Kriging does not converge to zero. Because of this, we adjust the estimated conditional variance of the stochastic Kriging by adding the limit value $\sigma^2(\Psi(x-x) - \Psi_S(x-x))$. The results are reported in Table \ref{tab:crci}.

% Because the length of confidence interval of stochastic Kriging for Gaussian process with only extrinsic error converges to zero, which does not reflect the fact that the actual MSPE of stochastic Kriging does not converge to zero,Uauthe 

% mycrKALE
% [1] 0.9184827 0.9128346 0.9092101 0.9302025
% > mycrKALEN
% [1] 0.9247594 0.9290089 0.9332758 0.9376765
% > mycrKALEnu
% [1] 0.8523772 0.7738670 0.6886089 0.5879278
% > mycrKALENnu
% [1] 0.4838195 0.4439595 0.4004137 0.3495713
% > mycrKALEnuad
% [1] 0.9675978 0.9619485 0.9603750 0.9391526
% > mycrKALENnuad
% [1] 0.6221047 0.6468629 0.6636683 0.6609711
% set.seed(43)
% > mycrKALE
% [1] 0.9178778 0.9267629 0.9201637 0.9162942
% > mycrKALEN
% [1] 0.9292038 0.9296100 0.9344882 0.9358430
% > mycrKALEnu
% [1] 0.8547382 0.7905699 0.6986852 0.5834171
% > mycrKALENnu
% [1] 0.4903337 0.4432208 0.4032571 0.3493651
% > mycrKALEnuad
% [1] 0.9630396 0.9754368 0.9670441 0.9213186
% > mycrKALENnuad
% [1] 0.6328384 0.6489539 0.6677403 0.6545382

\begin{table}[h]
\centering
\begin{tabular}{|c|c|c|c|c|c|c|}
\hline  
$\sigma_\epsilon^2$  & KALE  & SK$_1$ & Adjusted SK$_1$ & KALEN  & SK$_2$ & Adjusted SK$_2$\\
\hline
0.05 & 0.9179 & 0.8547 & 0.9630 & 0.9292 & 0.4903 & 0.6328\\
0.10 & 0.9268 & 0.7906 & 0.9754 & 0.9296 & 0.4432 & 0.6490\\
0.15 & 0.9202 & 0.6987 & 0.9670 & 0.9345 & 0.4033 & 0.6677\\
0.20 & 0.9163 & 0.5834 & 0.9213 & 0.9358 & 0.3494 & 0.6545\\
\hline
\end{tabular}
\caption{{Coverage rate of pointwise confidence interval of KALE and stochastic Kriging (when there is no noise on unobserved point), and KALEN and stochastic Kriging (when there is noise on unobserved point). The following notation is used: (Adjusted) SK$_1$ = (Adjusted) Stochastic Kriging without noise at the unobserved point; (Adjusted) SK$_2$ = (Adjusted) Stochastic Kriging with noise at the unobserved point. The nominal level is selected to be 95\%.}}
\label{tab:crci}
\end{table}

From Table \ref{tab:crci}, it can be seen that the (misspecified) pointwise confidence interval does not achieve the nominal level. It is expected that the stochastic Kriging has poor coverage because the model is misspecified. KALE and KALEN, on the other hand, can provide more reliable confidence interval. In fact, even for Guassian process without error, it is often observed that Gaussian process models have poor coverage of their confidence intervals \cite{gramacy2012cases,joseph2011regression,yamamoto2000alternative}. Therefore, a better uncertainty quantification methodology for Gaussian process with input location error is needed.
}

\subsection{Example 2}
In this example, we compare the calculation time of stochastic Kriging and KALE, where the predictor \eqref{BLUPNoN} of KALE does not have an analytic form. Suppose the underlying function is $f(x) = [(30+5x_1\sin(5x_1))(4+\exp(-5x_2))-100]/6$ for $x_1,x_2\in [0,1]$ \cite{lim2002design}. We use Mat\'ern correlation functions \cite{stein1999interpolation}
\begin{eqnarray}\label{matern}
	\Psi_M(x;\nu,\phi)=\frac{1}{\Gamma(\nu)2^{\nu-1}}(2\sqrt{\nu}\phi \|x\|_2)^\nu K_\nu(2\sqrt{\nu}\phi\|x\|_2)
\end{eqnarray}
to make predictions, where $K_\nu$ is the modified Bessel function of the second kind, and $\nu$ and $\phi$ are model parameters. The Mat\'ern correlation function can control the smoothness of the predictor by $\nu$ and thus is more robust than a Gaussian correlation function \cite{wang2019prediction}. The covariance function is chosen to be $\Psi(x-y) = \Psi_M(x-y;\nu,\phi)$. {The intrinsic noise is chosen to be mean zero normally distributed with the variances $0.01k$, for $k=2,3,4,5$.} We use maximin Latin hypercube design with 20 points to estimate parameters, and choose the first 100 points in the Halton sequence \cite{halton1964algorithm} as testing points. The smoothness parameter $\nu$ is chosen to be $3$, which can provide a robust estimator of $f$. {In order to improve the prediction performance, we use ordinary Kriging, where the mean in Gaussian process model is assumed to be an unknown constant instead of zero, i.e., $f$ is a realization of Gaussian process with unknown mean $\beta$ and covariance function $\sigma^2\Psi_M$.}

If we use a Mat\'ern correlation function, the integrals in \eqref{eq:RvectorNoi} and \eqref{eq:KernelMatrixNoi} do not have analytic forms and are calculated by Monte-Carlo sampling. We randomly choose 30 points to approximate the integral in \eqref{eq:RvectorNoi}, and 900 points to approximate the integral in \eqref{eq:KernelMatrixNoi}. Preliminary results show that, if we use  Monte-Carlo sampling with different points every time in the evaluation of the integrals in (\ref{eq:RvectorNoi}) and (\ref{eq:KernelMatrixNoi}), it is not possible to use maximum pseudo-likelihood estimation to estimate the unknown parameters, consisting of $\phi$ in \eqref{matern}, $\sigma^2$, the variance of noise $\sigma_\epsilon^2$ and the mean $\beta$. The reason is that at each step of the optimization in maximum pseudo-likelihood estimation, we need to calculate the integral, whose computational cost is high. Therefore, we generate 900 points and 30 points randomly at one time and use these 900 points and 30 points for evaluations of (\ref{eq:KernelMatrixNoi}) and (\ref{eq:RvectorNoi}), respectively. Then we use maximum pseudo-likelihood estimation to estimate the unknown parameters. {We run 20 simulations and compute the average processing time and the approximated MSPE $\frac{1}{100}\sum_{i=1}^{100}(f(x_i) - \hat f(x_i))^2$, where $\hat f$ is the KALE predictor, and $x_i$'s are testing points. }

For stochastic Kriging, we use (misspecified) maximum likelihood estimation to estimate the unknown parameters, which are $\phi$ in \eqref{matern}, $\sigma^2$, the nugget parameter $\mu$ and the mean $\beta$. {We run 100 simulations and compute the average processing time and the approximated MSPE $\frac{1}{100}\sum_{i=1}^{100}(f(x_i) - \hat f(x_i))^2$, where $\hat f$ is the stochastic Kriging predictor, and $x_i$'s are testing points.} The RMSPE, {which is the square root of MSPE,} and the processing time of KALE and stochastic Kriging are shown in Table \ref{tab:eg2dcasematern}.
%{\bf BH - I would be consistent in usage of RMSPE vs. MSPE throughout. RMSPE is more interpretable. WW - Yes, I have changed MSPE to RMSPE, except the case study, because it seems that SD is typical in the case study.}

% \begin{table}[h]
% \centering
% \begin{tabular}{|c|c|c|c|}
% \hline  
% $\sigma_\epsilon^2$  & RMSPE of KALE  & RMSPE of stochastic & Difference\\
%           &                &  Kriging &   \\
% \hline
% 0.02 & 1.516 & 1.860  & 0.344\\
% 0.03 & 1.045 & 1.615  & 0.570\\
% 0.04 & 1.652 & 2.372  & 0.720\\
% 0.05 & 1.764 & 2.429  & 0.665\\
% \hline
% \end{tabular}
% \caption{Comparison of the RMSPE of KALE and stochatic Kriging: 2-d function with Mat\'ern correlation function. In last column, difference $=$ 3rd column $-$ 2nd column. }
% \label{tab:eg2dcasematern}
% \end{table}

% \begin{table}[h]
% \centering
% \begin{tabular}{|c|c|c|c|c|c|}
% \hline  
% $\sigma_\epsilon^2$  & RMSPE(MAE) & Processing time & RMSPE(MAE) of  & Processing time &Difference\\
%           &    of KALE         & of KALE & stochastic  Kriging &   of stochastic Kriging  &  \\
% \hline
% 0.02 & 1.780(1.462) & 646.38  & 1.819(1.586)  & 6.85 & 0.124\\
% 0.03 & 1.091(0.982) & 884.84  & 1.779(1.314)  & 5.56 & 0.392\\
% 0.04 & 1.320(0.967) & 868.49  & 1.833(1.340)  & 5.73 & 0.435\\
% 0.05 & 2.270(1.740) & 1134.98 & 2.382(1.869)  & 5.25  & 0.129\\
% \hline
% \end{tabular}
% \caption{Comparison of the RMSPE of KALE and stochatic Kriging: 2-d function with Mat\'ern correlation function. The processing time is in seconds. In last column, difference $=$ 3rd column $-$ 2nd column. }
% \label{tab:eg2dcasematern}
% \end{table}

\begin{table}[h]
\centering
\begin{tabular}{|c|c|c|c|c|c|}
\hline  
$\sigma_\epsilon^2$  & RMSPE & PT & RMSPE of  & PT &Difference \\
           &    of KALE         & of KALE & SK &   of  SK &\\
\hline
0.02 & 1.5292 & 648.86 & 1.9852 & 0.6261   & 0.4559 \\
0.03 & 1.7899 & 633.55  & 2.2346 & 0.5947  & 0.4446\\
0.04 & 1.9734 & 695.27  &  2.5226 & 0.5848   & 0.5492\\
0.05 & 2.4501 & 748.33 & 3.3415 & 0.5803  & 0.8915\\
\hline
\end{tabular}
\caption{The RMSPE of KALE and stochatic Kriging: 2-d function with Mat\'ern correlation function. The processing time is in seconds. In sixth column, difference $=$ 4th column $-$ 2nd column, i.e., the RMSPE of stochastic Kriging $-$ the RMSPE of KALE. The following abbreviation is used: PT = Processing time, SK = stochastic Kriging.}
\label{tab:eg2dcasematern}
\end{table}

% 100 run

% PT of SK: 0.6261 0.5947 0.5848 0.5803

% RMSPE of SK: 1.985178 2.234568 2.522631 3.341530

% RMSPE of MCMC(1.2 * mysigma* runif(1): 2.427284 2.511527 2.408002 2.404914 2.413649

% PT of MCMC: 40.6641 40.5165 40.6822 40.4794 

% 20 run

% RMSPE of KALE: 1.529249 1.789926 1.973428 2.450072

% PT of KALE: 648.8595 633.5475 695.2710 748.3260

% diff 0.455929 0.444642 0.549203 0.891458

It can be seen that KALE has some improvement on prediction accuracy over stochastic Kriging. However, KALE takes too much computation time, even though the numbers of design points and testing points are relatively small. The comparison would get worse as the number of points became larger. Therefore, if the integrals in \eqref{eq:RvectorNoi} and \eqref{eq:KernelMatrixNoi} do not have analytic forms, stochastic Kriging is preferred, especially when the sample size is large and the variance of intrinsic noise is small.

% In Table \ref{tab:eg2dcasematern} it can be seen that as the variance of intrinsic noise increases, the difference of two methods becomes larger, which is behaves similarly in Example \ref{Eg1subsec}. Although the number of design points and testing points is small, it still takes a long time for calculation in KALE. 

% \begin{table}[h]
% \centering
% \begin{tabular}{|c|c|c|c|}
% \hline  
% $\sigma_1$  & MSPE of KALE  & MSPE of a nugget \\
%            &                &  under noise    \\
% \hline
% 0.01 & 0.8821766 & 0.3175058  \\
% 0.02 & 1.731895 & 0.4941639  \\
% 0.03 & 1.161202 & 0.8563318  \\
% 0.04 & 1.92337 & 0.3830478  \\
% 0.05 & 1.94202 & 1.356058  \\
% \hline
% \end{tabular}
% \caption{5-d function with noise}
% \label{tab:eg2dcasematern}
% \end{table}

\section{Case Study: Application to Composite Parts Assembly Process}\label{sec:casestudy}

To illustrate the performance of KALEN and stochastic Kriging, we apply them to a real case study, the composite parts assembly process. As shown in Figure \ref{P1caseFig} (a) and Figure \ref{P1caseFig}  (b), ten adjustable actuators are installed at the edge of a composite part \cite{wen2018FEA,yue2018surrogate}. These actuators can provide push or pull forces in order to adjust the shape of the composite part to the target dimensions. The dimensional shape adjustment of composite parts is one of the most important steps in the aircraft assembly process. 
It reduces the gap between the composite parts and decreases the assembly time with improved dimensional quality. Detailed descriptions about the shape adjustment of composite parts can be found in \cite{wen2018FEA}. Modeling of composite parts is the key for shape adjustment. The objective is to build a model that has the capability to predict the dimensional deviations accurately under specific actuators’ forces.  In this model, the input variables are ten actuators' forces. The responses are the dimensional deviations of multiple critical points along the edge plane near the actuators, shown in Figure \ref{P1caseFig}  (c). We consider responses at 91 critical points around the composite edge in the case study.  
\begin{figure}[h!]
    \centering
    \includegraphics[height=2in]{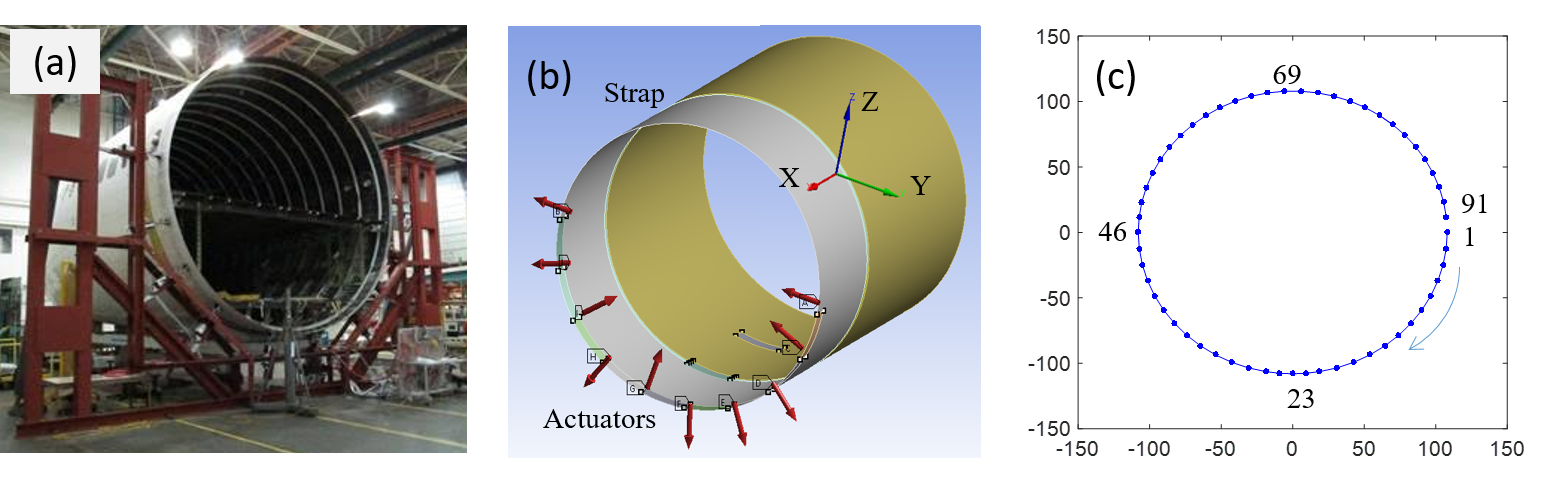}
    \caption{Schematic diagram for composite part shape adjustment: (a) composite part shape adjustment \cite{wen2018FEA}, (b) layout of ten actuators, (c) multiple critical points.}\label{P1caseFig}
\end{figure}

In the shape control of composite parts, intrinsic noise commonly exists in the actuators’ forces \cite{yue2018surrogate}. When a force is implemented by an actuator, the actual force may not be exactly the same as the target force. The magnitudes of forces may have uncertainties naturally due to the device tolerances of the hydraulic or electromechanical system of actuators. Uncertainties in the directions and application points of forces come from the deviations of contact geometry of actuators and their installations. For the modeling of composite parts, there are two steps: (i) training the parameters using experimental data; (ii) predicting dimensional deviations for new actuators' forces. In the training step, we need to consider input error in the experimental data. Additionally, when new actuator forces are implemented in practice, the uncertainty in the actual delivered forces inevitably exists. This suggests that KALEN is suitable for this application scenario. 
% \textit{WW - We rewrite these sentences to make it clear. }} 
% {\bf \{BH - Small updates above.\}}
% Therefore, KALEN and stochastic Kriging are two typical models that can be used in this situation.  
%We have described the theoretical properties of these methods, and we will show their performance in this case study.  
% Moreover, after we have the model for composite fuselage shape control, we need to use this model for prediction on an unobserved point with input error. Therefore, KALEN is closer to the real case than KALE. 
We will show the performance of KALEN and compare it with stochastic Kriging as follows. 

The model we use in this case study is  $Y^{(j)} = F^T\beta^{(j)} + Z^{(j)}(F)$ for $j=1,\ldots,91$, where $Y^{(j)}$ is the dimensional deviation vector of the composite part at the critical point $j$, {$F=(F_{(1)},\ldots,F_{(10)})^T\in \RR^{10}$ is the vector of actuators' forces, and $Z^{(j)}(\cdot)$ is a mean zero Gaussian process, with input variables in $\mathbb{R}^{10}$. The covariance of $Z^{(j)}(F_1)$ and $Z^{(j)}(F_2)$ for any forces $F_1=(F_{1,(1)},\ldots,F_{1,(10)})^T$ and $F_2=(F_{2,(1)},\ldots,F_{2,(10)})^T$ is assumed to be $\sigma_j^2\exp(-\sum_{k=1}^{10}\theta_{jk}(F_{1,(k)} - F_{2,(k)})^2)$, where $\sigma_j,\theta_{jk} > 0$ are parameters. We assume the intrinsic noise $\epsilon\sim N(0,\sigma_\epsilon^2I_{10})$, where $N(0,\sigma_\epsilon^2I_{10})$ is a mean zero normal distribution with covariance matrix $\sigma_\epsilon^2I_{10}$.} The parameters $\beta^{(j)}$, $\theta_{jk}$, $\sigma_\epsilon^2$, and $\sigma_j^2$ are estimated by maximum (pseudo-)likelihood estimation as described in Section \ref{sec:ParaEst}. The mean function $F^T\beta^{(j)}$ we use in this model is to represent the linear component in dimensional shape control of composite fuselage, which follows the approach in \cite{yue2018surrogate}. Specifically, according to the mechanics of composite material and classical lamination theory, there is a linear relationship between dimensional deviations and actuators' forces within the elastic zone. The term $F^T\beta^{(j)}$ describes how the actuators' forces impact the part deviations linearly, and $Z^{(j)}(\cdot)$ represents the nonlinear components so as to obtain accurate predictions.

For the computer experiments, we generated 50 training samples and 30 testing samples based on a maximin Latin hypercube design. The designed experiments are conducted in the finite element simulation platform developed by \cite{wen2018FEA}. It is worth mentioning that the computer simulation here is not a deterministic simulation {because we add the intrinsic noise at the input points in simulation to simulate the randomness in the real process. Therefore, repeated runs with the same input points will have different outputs.} The intrinsic noise is added to the actuators' forces to mimic real actuators. The standard deviations (SD) of actuators' forces are chosen to be 0.005, 0.01, 0.02, 0.03, and 0.04 lbf (lbf is a unit of pound-force), which is determined by the tolerance of different kinds of actuators according to engineering domain knowledge. The maximum actuators' force is set to 600 lbf. After we have the computer experiment data, we can estimate the parameters of KALEN by solving the pseudo-likelihood equation \eqref{pseudomleestimator}, and the parameters of stochastic Kriging by solving the maximum likelihood equation \eqref{MLE}. 
% {\bf BH - Should this be (4.2)? WW - We use \eqref{MLE} for estimate parameters for stochastic Kriging} 
Then, we can use the model to predict dimensional deviations at the unobserved points in the testing dataset.  

The performance of KALEN and stochastic Kriging are compared in terms of mean absolute error (MAE). This is an index that has been commonly used in the composite parts assembly domain to evaluate the modeling performance. We also compare RMSPE of KALEN and stochastic Kriging, {and the processing time of generating each output. The RMSPE is the square root of MSPE, which is approximated by the average of $\frac{1}{30}\sum_{i=1}^{30}(Y^{(j)}(F_i) - \hat Y^{(j)}(F_i))^2$ on the 91 points, where $F_i$'s are the inputs of testing samples, $Y^{(j)}(F_i)$ is the observed testing data, and $\hat Y^{(j)}(F_i)$ is the KALEN predictor. The MAE is approximated by $\frac{1}{30}\sum_{i=1}^{30}|Y^{(j)}(F_i) - \hat Y^{(j)}(F_i)|$ on the 91 points.} %The MAE and RMSPE are approximated by averaging the error on the 91 points and multiple samples.
%{\bf BH - The differences in the table below might be better reported in scientific notation with 3 digits of accuracy as, for example, $1.23\times 10^{-4}$. WW - I have changed it.}

% \begin{table}[h]
% \centering
% \begin{tabular}{|c|c|c|c|}
% \hline  
% SD of  & MAE   & MAE   &  Difference \\
%   actuators' forces   &      of KALEN          & of stochastic Kriging &   \\
% \hline
% 0.005 & 0.0059 & 0.0059 & -9.5102e-06\\
% 0.01 & 0.0117 & 0.0119  & 1.7167e-04 \\ %15 & 1.318 & 1.315  & 0.003 & 54\\
% 0.02 & 0.0216 & 0.0217  & 9.4504e-05 \\ %25 & 1.301 & 1.313  & 0.012 & 64\\
% 0.03 & 0.0286 & 0.0304  & 0.0017 \\
% 0.04 & 0.0468 & 0.0486  & 0.0018\\
% \hline
% \end{tabular}
% \caption{The MAE of KALEN and stochastic Kriging in the composite part modeling}
% \label{tab:varchangeBoeing}
% \end{table}

\begin{table}[h]
\centering
\begin{tabular}{|c|c|c|c|c|c|}
\hline  
SD of & MAE (RMSPE)   & MAE (RMSPE)    &  Difference   & PT of &PT of \\
    AF &      of KALEN          & of SK &    & KALEN & SK\\
\hline
0.005 & 0.0059 (0.0081) & 0.0059 (0.0081) & $7.1 \times 10^{-7}$ ($1.9 \times 10^{-6}$)  & 0.1500   &   0.3415 \\
0.01 & 0.0117 (0.0147) & 0.0119 (0.0151)  & $1.7 \times 10^{-4}$ ($3.7\times 10^{-4}$)   & 0.4691 &	0.3938\\ %15 & 1.318 & 1.315  & 0.003 & 54\\
0.02 & 0.0216 (0.0265) & 0.0217 (0.0264)  & $9.5 \times 10^{-5}$ ($-8.7\times 10^{-5}$)   & 0.5048  & 0.3964\\ %25 & 1.301 & 1.313  & 0.012 & 64\\
0.03 & 0.0286 (0.0335) & 0.0304 (0.0376)  & $1.7\times 10^{-3}$ ($4.1\times 10^{-3}$)   & 0.6746 & 0.4115\\
0.04 & 0.0389 (0.0478) & 0.0486 (0.0610)  & $9.7\times 10^{-3}$ ($1.3\times 10^{-2}$)  &  0.6529 & 0.4302\\
\hline
\end{tabular}
\caption{The MAE (RMSPE) of KALEN and stochastic Kriging in the composite part modeling. In 4th column, difference $=$ 3rd column $-$ 2nd column. The processing time is in seconds. The following abbreviation is used:  AF = actuators' forces, PT = Processing time for each output, SK = stochastic Kriging.}
\label{tab:varchangeBoeing}
\end{table}

% SK 0.1500	0.4691	0.5048	0.6746	0.6529
% KALE: 0.3415	0.3938	0.3964	0.4115	0.4302

The MAE and RMSPE of KALEN and stochastic Kriging are summarized in Table \ref{tab:varchangeBoeing}. As the  SD of actuators' forces changes from 0.04 lbf to 0.005 lbf, the MAE and RMSPE of KALEN and stochastic Kriging also decrease. This result is consistent with the conclusions in Theorem \ref{Thm:main} and Proposition \ref{Upprop}. The MAE and RMSPE of KALEN are slightly smaller than the MAE and RMSPE of stochastic Kriging. Generally speaking, their performances are comparable, especially when the SD of actuators' forces is small. The main reason is that, when the uncertainty in the input variables is small, stochastic Kriging can approximate the best linear unbiased predictor KALEN very well. Since a Gaussian correlation function is used, the computational complexity of KALEN and stochastic Kriging are the same. {The computation time of KALEN is smaller than that of the stochastic Kriging in this example. We conjecture this is because of the different computation time of maximum (pseudo-) likelihood estimation.} In summary, if high-quality actuators are used and the intrinsic noise in the actuators is therefore small, then both KALEN and stochastic Kriging can realize very good prediction performance. When the intrinsic noise in the actuators' forces becomes larger, KALEN outperforms stochastic Kriging.

\section{Conclusions and Discussion}\label{sec:conclusion}
We first summarize our contributions in this work. We have investigated three predictors, KALE, KALEN and stochastic Kriging, as applied to Gaussian processes with input location error. When predicting the mean Gaussian process output at an unobserved point with intrinsic noise, we prove that the limits of MSPE of KALEN and stochastic Kriging are the same as the fill distance of the design points goes to zero. If there is no noise at an unobserved point, we provide an upper bound on the MSPE of KALE and stochastic Kriging. The upper bound is close to zero if the noise is small, which implies the MSPE of KALE and stochastic Kriging are close. We also provide an asymptotic upper bound on the MSPE of KALE/KALEN and stochastic Kriging with estimated parameters. These results indicate that if the number of data points is large or the variance of the intrinsic noise is small, then there is not much difference between KALE/KALEN and stochastic Kriging in terms of prediction accuracy. The numeric results corroborate our theory. A case study is presented to illustrate the performance of KALEN and stochastic Kriging for modeling in the composite parts assembly process. 

The calculation of the predictor \eqref{BLUPNoN} is not efficient if the integrals in \eqref{eq:RvectorNoi} and \eqref{eq:KernelMatrixNoi} do not have an analytic form. If the sample size is large, then using pseudo maximum likelihood to estimate the unknown parameters is challenging, especially when the integrals in \eqref{eq:RvectorNoi} and \eqref{eq:KernelMatrixNoi} do not have analytic forms. In this case, using stochastic Kriging as an alternative would be more desirable.

{There are several problems remain to be solved.} In this paper, the MSPE of KALE, KALEN, and stochastic Kriging are primarily considered asymptotically, {i.e., the number of design points goes to infinity}. The theory does not cover the results under non-asymptotic cases, {i.e., the number of design points is fixed}. It can be expected that the difference between the MSPE of KALE/KALEN and stochastic Kriging will decrease as the fill distance decreases. {If there is no noise on an unobserved point, only upper bounds are obtained for KALE and stochastic Kriging. The asymptotic performance of KALE and stochastic Kriging when unobserved point has no noise will be pursued in the future work.}

%{\bf In sum, our theoretical and numerical results indicate that if the number of data points is large, the integrals in \eqref{eq:RvectorNoi} and \eqref{eq:KernelMatrixNoi} do not have an analytic form, and/or the variance of noise is small, then using stochastic Kriging is desirable. WW - Is it redundant?}

\appendix

\section{Reproducing Kernel Hilbert Space}\label{app:introrkhs}
In this section we briefly introduce the reproducing kernel Hilbert space used in Assumption \ref{assumpsi}. For detailed introduction to the reproducing kernel Hilbert space, we refer to \cite{wendland2004scattered}. One way to define the reproducing kernel Hilbert space is via Fourier transform, defined by
$$\mathcal{F}(f)(\omega)=(2\pi)^{-d/2}\int_{\mathbb{R}^d} f(x) e^{-ix^T\omega}d x$$ for $f\in L_1(\mathbb{R}^d)$. The definition of the reproducing kernel Hilbert space can be generalized to $f\in L_2(\RR^d)\cap C(\RR^d)$. See \cite{girosi1995regularization} and Theorem 10.12 of \cite{wendland2004scattered}.

%Let $L_2(\RR^d)$ be the space of complex-valued square integrable functions on $\RR^d$, and $C(\RR^d)$ be the space of continuous real-valued functions on $\RR^d$. 

\begin{defn}\label{Def:NativeSpace}
Let $\Psi\in L_1(\RR^d)\cap C(\RR^d)$ be a positive definite function. Define the reproducing kernel Hilbert space $\mathcal{N}_\Psi(\RR^d)$ generated by $\Psi$ as
	$$\mathcal{N}_\Psi(\RR^d):=\{f\in L_2(\RR^d)\cap C(\RR^d):\mathcal{F}(f)/\sqrt{\mathcal{F}(\Psi)}\in L_2(\RR^d)\},$$
	with the inner product
	$$\langle f,g\rangle_{\mathcal{N}_\Psi(\RR^d)}=(2\pi)^{-d/2}\int_{\RR^d}\frac{\mathcal{F}(f)(\omega)\overline{\mathcal{F}(g)(\omega)}}{\mathcal{F}(\Psi)(\omega)}d \omega.$$
\end{defn}

By Bochner's theorem (Page 208 of \cite{gihman1974theory}; Theorem 6.6 of \cite{wendland2004scattered}) and Theorem 6.11 of \cite{wendland2004scattered}, if $\Psi$ is a correlation function (thus positive definite), there exists a density function $f_\Psi$ such that 
\begin{align*}%\label{Bochner}
\Psi(x)=\int_{\mathbb{R}^d} e^{i \omega^T x}  f_\Psi(\omega) d \omega
\end{align*}
for any $x\in \RR^d$. The function $f_\Psi$ is known as the \textit{spectral density} of $\Psi$. 

For a positive number $\nu > d/2$, the Sobolev space on $\RR^d$ with smoothness $\nu$ can be defined as
\begin{align*}
H^\nu(\mathbb{R}^d) = \{f\in L_2(\mathbb{R}^d): |\mathcal{F}(f)(\cdot)| (1+\|\cdot\|_2^2)^{\nu/2}\in L_2(\mathbb{R}^d)\},
\end{align*}
equipped with an inner product
$$\langle f,g\rangle_{H^\nu(\mathbb{R}^d)}=(2\pi)^{-d/2}\int_{\RR^d}\mathcal{F}(f)(\omega)\overline{\mathcal{F}(g)(\omega)}(1+\|\omega\|_2^2)^{\nu}d \omega.$$ It can be shown that $H^\nu(\mathbb{R}^d)$ coincides with the reproducing kernel Hilbert space $\mathcal{N}_\Psi(\RR^d)$, if $\Psi$ satisfies Condition \ref{C1} (\cite{wendland2004scattered}, Corollary 10.13).

\begin{condition}\label{C1}
	There exist constants $c_2 \geq c_1>0$ and $\nu>0$ such that, for all $\omega\in\mathbb{R}^d$,
	$$ c_1(1+\|\omega\|_2^2)^{-\nu} \leq  f_\Psi(\omega)\leq c_2(1+\|\omega\|_2^2)^{-\nu}. $$
\end{condition}
The isotropic Mat\'ern correlation function \eqref{matern} has the spectral density \cite{tuo2015theoretical}
\begin{eqnarray*}
f_{\Psi_M}(\omega;\nu,\phi)= \pi^{-d/2}\frac{\Gamma(\nu+d/2)}{\Gamma(\nu)}(4\nu \phi^2)^{\nu} (4\nu\phi^2+\|\omega\|_2^2)^{-(\nu+d/2)}.
\end{eqnarray*}
We can see $\Psi_M$ satisfies Condition \ref{C1}. Thus, the reproducing kernel Hilbert space generated by $\Psi_M$ coincides with the Sobolev space $H^{\nu +d/2}$, which implies $\Psi_M$ fulfills Assumption \ref{assumpsi}.

The isotropic Gaussian correlation function $\Psi_G(x) = e^{-\theta \|x\|^2}$ has the spectral density (Theorem 5.20 of \cite{wendland2004scattered})
\begin{eqnarray*}
    f_{\Psi_G}(\omega) = (4\pi\theta)^{-d/2}e^{-\|\omega\|_2^2/(4\theta)}.
\end{eqnarray*}
Since for any fixed $\nu$, $f_{\Psi_G}(\omega)\leq C(1+\|\omega\|_2^2)^{-\nu-d/2}$ for some constant $C$ not depending on $\omega$, the reproducing kernel Hilbert space generated by $\Psi_G$ can be embedded the Sobolev space $H^{\nu+d/2}(\mathbb{R}^d)$. This implies $\Psi_G$ fulfills Assumption \ref{assumpsi}.

A reproducing kernel Hilbert space can also be defined on a suitable subset (for example, convex and compact) $\Omega\subset \RR^d$, denoted by $\mathcal{N}_\Psi(\Omega)$, with norm
\begin{eqnarray*}%\label{restriction}
\|f\|_{\mathcal{N}_\Psi(\Omega)}=\inf\{\|f_E\|_{\mathcal{N}_\Psi(\RR^d)}:f_E\in\mathcal{N}_\Psi(\RR^d),f_E|_\Omega=f\},
\end{eqnarray*}
where $f_E|_\Omega$ denotes the restriction of $f_E$ to $\Omega$. A Sobolev space on $\Omega$ can be defined in a similar way. By the extension theorem \cite{devore1993besov}, the reproducing kernel Hilbert space defined on space $\Omega$ generated by $\Psi_M$ and $\Psi_G$ can be embedded into the Sobolev space $H^{\nu+d/2}(\Omega)$.
}

\section{A Lemma about MSPE of Stochastic Kriging}\label{app:A}

\begin{lemma}\label{lemma2skzero}
Assume Assumptions \ref{assumpsi} and \ref{assumX} are true for a positive definite function $\Psi$ and a sequence of designs $X=\{x_1,\ldots,x_n\}$. Then for any fixed constant $\mu>0$, $\Psi(x-x) - \Psi(x-X)(\Psi(X-X) + \mu I)^{-1}\Psi(x-X)^T$ converges to zero as the fill distance of $X$ goes to zero, where $\Psi(x-X) = (\Psi(x-x_1),\ldots,\Psi(x-x_n))^T$ and $\Psi(X-X) = (\Psi(x_j-x_k))_{jk}$. 
\end{lemma}
\begin{proof}
{
Let $\bar X=\{\bar x_1,...,\bar x_m\}$ be the distinct design points corresponding to $X$. At each design point $\bar x_j\in \bar X$, suppose there are $a_j$ replicates, thus,
\begin{gather}
X=\{\underbrace{\bar x_1,\ldots,\bar x_1}_{a_1\;{\rm replications}},\underbrace{\bar x_2,\ldots,\bar x_2}_{a_2\;{\rm replications}},\ldots, \underbrace{\bar x_m,\ldots,\bar x_m}_{a_m\;{\rm replications}}\}. \nonumber
\end{gather}
It can be shown that $\Psi(x-x) - \Psi(x-X)(\Psi(X-X) + \mu I)^{-1}\Psi(x-X)^T = \Psi(x-x) - \Psi(x-\bar X)(\Psi(\bar X-\bar X) + \Lambda I)^{-1}\Psi(x-\bar X)^T$, where $\Lambda={\rm diag}(\lambda_1,...,\lambda_m)$ with $\lambda_j=\mu/a_j$ (See Lemma 3.1 of \cite{binois2018practical} and the proof of Proposition 3.1 of \cite{wang2018controlling}). Let $a=\min_j a_j$. We have
\begin{align*}
    & \Psi(x-x) - \Psi(x-X)(\Psi(X-X) + \mu I)^{-1}\Psi(x-X)^T\\  = & \Psi(x-x) - \Psi(x-\bar X)(\Psi(\bar X-\bar X) + \Lambda I)^{-1}\Psi(x-\bar X)^T \\
   \leq  & \Psi(x-x) - \Psi(x-\bar X)(\Psi(\bar X-\bar X) + \mu/a I)^{-1}\Psi(x-\bar X)^T\\
   \leq & \|\Psi(\cdot-x) - \Psi(\cdot-\bar X)^T(\Psi(\bar X-\bar X) + \mu/a I)^{-1}\Psi(x-\bar X)^T\|_{L_\infty(\Omega)},
\end{align*}
where the first inequality is because $(\Psi(\bar X-\bar X) + \Lambda I)^{-1} \succeq (\Psi(\bar X-\bar X) + \mu/a I)^{-1}$. Here $A\succeq B$ denotes that for any vector $b$, $b^T(A-B)b\geq 0$.}

{ Define $g(t) = \Psi(t-x) - \Psi(t-\bar X)(\Psi(\bar X-\bar X) + \mu/a I)^{-1}\Psi(x-\bar X)^T$. Under Assumption \ref{assumpsi}, we have $g\in H^\nu(\Omega)$, where $H^\nu(\Omega)$ is the Sobolev space with smoothness $\nu$. By the interpolation inequality, $\|g\|_{L_\infty(\Omega)}\leq C_1\|g\|_{L_2(\Omega)}^{1 - \frac{d}{2\nu}}\|g\|_{H^\nu(\Omega)}^{\frac{d}{2\nu}}$. By Corollary 10.25 in \cite{wendland2004scattered} and the fact that $\Psi(\bar X-\bar X)^{-1}\succeq(\Psi(\bar X-\bar X) + \mu/a I)^{-1} $, it can be shown that
\begin{align*}
    & \|g\|_{H^\nu(\Omega)}\leq C_2 \|g\|_{\mathcal{N}_{\Psi}(\Omega)}\\ 
    \leq & C_2 (\Psi(x-x) - 2\Psi(x-\bar X)(\Psi(\bar X-\bar X) + \mu/a I)^{-1}\Psi(x-\bar X)^T\\ 
    & + \Psi(x-\bar X)(\Psi(\bar X-\bar X) + \mu/a I)^{-1}\Psi(\bar X-\bar X)(\Psi(\bar X-\bar X) + \mu/a I)^{-1}\Psi(x-\bar X)^T) \\
    \leq & C_2 (\Psi(x-x) - \Psi(x-\bar X)(\Psi(\bar X-\bar X) + \mu/a I)^{-1}\Psi(x-\bar X)^T) \leq C_2\Psi(x-x), 
\end{align*}
where $\|g\|_{\mathcal{N}_{\Psi}(\Omega)}$ is the norm of $g$ in the reproducing kernel Hilbert space $\mathcal{N}_{\Psi}(\Omega)$. Thus, the result follows if we can show $\|g\|_{L_2(\Omega)}$ converges to zero. By the representer theorem, $\hat g_1(t):=\Psi(t-\bar X)(\Psi(\bar X-\bar X) + \mu/a I)^{-1}\Psi(x-\bar X)^T$ is the solution to the optimization problem
\begin{align}\label{lemma2a1}
    \min_{g_1\in \mathcal{N}_{\Psi}(\Omega)} \frac{1}{n}\sum_{j=1}^n (g_1(\bar x_j) - \Psi(x-\bar x_j))^2 +\frac{\mu}{an} \|g_1\|_{\mathcal{N}_{\Psi}(\Omega)}^2.
\end{align}
Note $g(t) = \Psi(t-x) - \hat g_1(t)$. Under Assumption \ref{assumX}, by Lemma 3.4 of \cite{utreras1988convergence}, the result follows from
\begin{align*}
    \|g\|_{L_2}^2 \leq & C_3 \bigg(\frac{1}{n}\sum_{j=1}^n (\hat g_1(\bar x_j) - \Psi(x-\bar x_j))^2 + h_{\bar X}^{2\nu} \|g\|_{H^\nu(\Omega)}^2\bigg)\\
    \leq & C_3 \bigg(\frac{1}{n}\sum_{j=1}^n (\hat g_1(\bar x_j) - \Psi(x-\bar x_j))^2 +\frac{\mu}{an} \|\hat g_1\|_{\mathcal{N}_{\Psi}(\Omega)}^2 + h_{\bar X}^{2\nu} \|g\|_{H^\nu(\Omega)}^2\bigg)\\
    \leq & C_3 \bigg(\frac{1}{n}\sum_{j=1}^n ( \Psi(x-\bar x_j) - \Psi(x-\bar x_j))^2 +\frac{\mu}{an} \|\Psi(x-\cdot)\|_{\mathcal{N}_{\Psi}(\Omega)}^2 + h_{\bar X}^{2\nu} \|g\|_{H^\nu(\Omega)}^2\bigg)\rightarrow 0,
\end{align*}
where the last inequality is true because $\hat g_1$ is the solution to \eqref{lemma2a1}.}
\end{proof}

\section{Calculation of \eqref{eq:KernelMnormal}}\label{app:eq8}
{ In this section, we show that if the correlation function is $\Psi(s-t) = \exp(-\theta\|s - t\|_2^2)$, and the noise $\epsilon \sim N(0,\sigma_\epsilon^2I_d)$, where $\theta > 0$ is the correlation parameter, and $N(0,\sigma_\epsilon^2I_d)$ is the mean zero normal distribution with covariance matrix $\sigma_\epsilon^2I_d$, then (\ref{eq:RvectorNoi})--(\ref{eq:RvectorNoiX}) can be calculated respectively as in \eqref{eq:KernelMnormal}. Let $p_N(t)$ be the probability density function of normal distribution $N(0,\sigma_\epsilon^2I_d)$, i.e.,
\begin{align*}
    p_N(t) = \frac{1}{\sqrt{(2\pi\sigma_\epsilon^2)^d}}\exp\left(-\frac{t^Tt}{2\sigma_\epsilon^2}\right).
\end{align*}
The idea of calculating (\ref{eq:RvectorNoi})--(\ref{eq:RvectorNoiX}) is to utilize 
\begin{align*}
    \int_{\RR^d} \frac{1}{(2\pi a^2)^{d/2}} \exp\bigg(-\frac{\|s-b\|_2^2}{2a^2}\bigg) ds = 1
\end{align*}
for $a>0$ multiple times. By direct calculation, we have
\begin{align}\label{appeq8rN1}
    r_N(x,x_j) = & \sigma^2\int_{\RR^d}\int_{\RR^d}\Psi(x+\epsilon-(x_j+\epsilon_j))p(\epsilon_j)p(\epsilon)d\epsilon_jd\epsilon\nonumber\\
     = & \sigma^2\int_{\RR^d}\int_{\RR^d} \exp(-\theta\|x+\epsilon-(x_j+\epsilon_j)\|_2^2) \frac{1}{\sqrt{(2\pi\sigma_\epsilon^2)^d}}\exp\left(-\frac{\epsilon_j^T\epsilon_j}{2\sigma_\epsilon^2}\right)\frac{1}{\sqrt{(2\pi\sigma_\epsilon^2)^d}}\exp\left(-\frac{\epsilon^T\epsilon}{2\sigma_\epsilon^2}\right)d\epsilon_jd\epsilon\nonumber\\
     = & \sigma^2 \frac{\exp(-\theta \|x-x_j\|_2^2)}{(2\pi\sigma_\epsilon^2)^d} \int_{\RR^d}\int_{\RR^d}  \exp\left(-\bigg(\theta + \frac{1}{2\sigma_\epsilon^2}\bigg)\epsilon^T\epsilon-2\theta(x-x_j-\epsilon_j)^T\epsilon\right)d\epsilon\nonumber\\
     & \times \exp\left(-\bigg(\theta + \frac{1}{2\sigma_\epsilon^2}\bigg)\epsilon_j^T\epsilon_j+2\theta(x-x_j)^T\epsilon_j\right) d\epsilon_j.
\end{align}
We first compute
\begin{align}\label{appeq8rN2}
    & \int_{\RR^d} \exp\left(-\bigg(\theta + \frac{1}{2\sigma_\epsilon^2}\bigg)\epsilon^T\epsilon-2\theta(x-x_j-\epsilon_j)^T\epsilon\right)d\epsilon\nonumber\\
    = & \int \exp\bigg(-\bigg(\theta + \frac{1}{2\sigma_\epsilon^2}\bigg)\left\|\epsilon + \frac{\theta(x-x_j-\epsilon_j)}{\theta + \frac{1}{2\sigma_\epsilon^2}}\right\|_2^2 + \frac{\theta^2}{\bigg(\theta + \frac{1}{2\sigma_\epsilon^2}\bigg)}\|x-x_j-\epsilon_j\|_2^2\bigg)d\epsilon\nonumber\\
    = & \exp\bigg(\frac{2\sigma_\epsilon^2\theta^2}{1+2\sigma_\epsilon^2\theta}\|x-x_j-\epsilon_j\|_2^2\bigg) \sqrt{\bigg(2\pi\frac{\sigma_\epsilon^2}{1+2\theta\sigma_\epsilon^2}\bigg)^d}.
\end{align}
Plugging \eqref{appeq8rN2} into \eqref{appeq8rN1} yields
\begin{align}\label{appeq8rN3}
    r_N(x,x_j) = & \sigma^2 \frac{\exp(-\theta \|x-x_j\|_2^2)}{(2\pi\sigma_\epsilon^2)^d} \sqrt{\bigg(2\pi\frac{\sigma_\epsilon^2}{1+2\theta\sigma_\epsilon^2}\bigg)^d} \int_{\RR^d} \exp\bigg(\frac{2\sigma_\epsilon^2\theta^2}{1+2\sigma_\epsilon^2\theta}\|x-x_j-\epsilon_j\|_2^2\bigg) \nonumber\\
    & \times \exp\left(-\bigg(\theta + \frac{1}{2\sigma_\epsilon^2}\bigg)\epsilon_j^T\epsilon_j+2\theta(x-x_j)^T\epsilon_j\right) d\epsilon_j.
\end{align}
We next compute 
\begin{align}\label{appeq8rN4}
    & \int_{\RR^d} \exp\bigg(\frac{2\sigma_\epsilon^2\theta^2}{1+2\sigma_\epsilon^2\theta}\|x-x_j-\epsilon_j\|_2^2\bigg) \exp\left(-\bigg(\theta + \frac{1}{2\sigma_\epsilon^2}\bigg)\epsilon_j^T\epsilon_j+2\theta(x-x_j)^T\epsilon_j\right) d\epsilon_j\nonumber\\
    = & \exp\bigg(\frac{2\sigma_\epsilon^2\theta^2}{1+2\sigma_\epsilon^2\theta}\|x-x_j\|_2^2\bigg)  \int_{\RR^d} \exp\bigg(-\bigg( \theta + \frac{1}{2\sigma_\epsilon^2} - \frac{2\sigma_\epsilon^2\theta^2}{1+2\sigma_\epsilon^2\theta} \bigg)\epsilon_j^T\epsilon_j + 2\bigg(\theta - \frac{2\sigma_\epsilon^2\theta^2}{1+2\sigma_\epsilon^2\theta}\bigg)(x-x_j)^T\epsilon_j\bigg)d\epsilon_j\nonumber\\
    = & \exp\bigg(\frac{2\sigma_\epsilon^2\theta^2}{1+2\sigma_\epsilon^2\theta}\|x-x_j\|_2^2\bigg)  \int_{\RR^d} \exp\bigg(-\bigg( \frac{1+4\sigma_\epsilon^2\theta}{(1+2\sigma_\epsilon^2\theta)\sigma_\epsilon^2} \bigg)\epsilon_j^T\epsilon_j +\frac{2\theta}{1+2\sigma_\epsilon^2\theta}(x-x_j)^T\epsilon_j\bigg)d\epsilon_j\nonumber\\
    = & \exp\bigg(\frac{2\sigma_\epsilon^2\theta^2}{1+2\sigma_\epsilon^2\theta}\|x-x_j\|_2^2\bigg) \sqrt{\bigg(2\pi\frac{(1+2\sigma_\epsilon^2\theta)\sigma_\epsilon^2}{1+4\sigma_\epsilon^2\theta}\bigg)^d} \exp\bigg( \frac{(1+2\sigma_\epsilon^2\theta)\sigma_\epsilon^2}{1+4\sigma_\epsilon^2\theta} \frac{\theta^2}{(1+2\sigma_\epsilon^2\theta)^2}\|x-x_j\|_2^2 \bigg).
\end{align}
By plugging \eqref{appeq8rN4} into \eqref{appeq8rN3}, we obtain
\begin{align}\label{appeq8rN}
    r_N(x,x_j) = & \sigma^2 \frac{\exp(-\theta \|x-x_j\|_2^2)}{(2\pi\sigma_\epsilon^2)^d} \sqrt{\bigg(2\pi\frac{\sigma_\epsilon^2}{1+2\theta\sigma_\epsilon^2}\bigg)^d}\nonumber\\
    & \times \exp\bigg(\frac{2\sigma_\epsilon^2\theta^2}{1+2\sigma_\epsilon^2\theta}\|x-x_j\|_2^2\bigg) \sqrt{\bigg(2\pi\frac{(1+2\sigma_\epsilon^2\theta)\sigma_\epsilon^2}{1+4\sigma_\epsilon^2\theta}\bigg)^d} \exp\bigg( \frac{2(1+2\sigma_\epsilon^2\theta)\sigma_\epsilon^2}{1+4\sigma_\epsilon^2\theta} \frac{\theta^2}{(1+2\sigma_\epsilon^2\theta)^2}\|x-x_j\|_2^2 \bigg)\nonumber\\
    = & \frac{\sigma^2}{(1+4\sigma^2_\epsilon\theta)^{d/2}}\exp\bigg(\frac{-\theta\|x-x_j\|_2^2}{1+4\sigma^2_\epsilon\theta}\bigg),
\end{align}
which is desired. The term $r(x,x_j)$ can be computed by
\begin{align}\label{appeq8r}
    r(x,x_j) = & \sigma^2\int\Psi(x-(x_j+\epsilon_j))p(\epsilon_j)d\epsilon_j\nonumber\\
     = & \sigma^2\int_{\RR^d} \exp(-\theta\|x-(x_j+\epsilon_j)\|_2^2) \frac{1}{\sqrt{(2\pi\sigma_\epsilon^2)^d}}\exp\left(-\frac{\epsilon_j^T\epsilon_j}{2\sigma_\epsilon^2}\right)\nonumber\\
     = & \sigma^2 \frac{\exp(-\theta \|x-x_j\|_2^2)}{\sqrt{(2\pi\sigma_\epsilon^2)^d}} \int_{\RR^d}\exp\left(-\bigg(\theta + \frac{1}{2\sigma_\epsilon^2}\bigg)\epsilon_j^T\epsilon_j+2\theta(x-x_j)^T\epsilon_j\right) d\epsilon_j\nonumber\\
     = & \sigma^2 \frac{\exp(-\theta \|x-x_j\|_2^2)}{\sqrt{(2\pi\sigma_\epsilon^2)^d}} \exp\bigg(\frac{2\sigma_\epsilon^2\theta^2}{1+2\sigma_\epsilon^2\theta}\|x-x_j\|_2^2\bigg) \sqrt{\bigg(2\pi\frac{\sigma_\epsilon^2}{1+2\theta\sigma_\epsilon^2}\bigg)^d}\nonumber\\
     = & \frac{\sigma^2}{(1+2\sigma^2_\epsilon\theta)^{d/2}}\exp\bigg(\frac{-\theta\|x-x_j\|_2^2}{1+2\sigma^2_\epsilon\theta}\bigg).
\end{align}
Note $K_{jk} = r_N(x_j,x_k)$ if $j\neq k$. Together with \eqref{appeq8rN} and \eqref{appeq8r}, we obtain \eqref{eq:KernelMnormal}.}

\section{Proof of Lemma \ref{lemma1}}\label{App:pflemma}
By Fourier transform \cite{wendland2004scattered}, we have
\begin{align}\label{Fourpsi}
\Psi(x_j-x_k) = \frac{1}{(2\pi)^{d/2}}\int_{\mathbb{R}^d} e^{i\langle x_j-x_k,t \rangle}\mathcal{F}(\Psi)(t)dt,
\end{align}
where $\langle s,t \rangle = s^Tt$ is the inner product in $\mathbb{R}^d$. Therefore, by Fubini's theorem, direct calculation leads to
\begin{align}\label{eq:Psi1Fourier}
\Psi_S(x_j - x_k) & = \int_{\mathbb{R}^d}\int_{\mathbb{R}^d} \frac{1}{(2\pi)^{d/2}}\int_{\mathbb{R}^d}e^{i\langle x_j+\epsilon_1-(x_k+\epsilon_2),t \rangle}\mathcal{F}(\Psi)(t)p(\epsilon_1)p(\epsilon_2)dt d\epsilon_1 d\epsilon_2\nonumber \\
& = \frac{1}{(2\pi)^{d/2}}\int_{\mathbb{R}^d}\bigg(\int_{\mathbb{R}^d}\int_{\mathbb{R}^d} e^{i\langle x_j+\epsilon_1-(x_k+\epsilon_2),t \rangle}p(\epsilon_1)p(\epsilon_2)d\epsilon_1 d\epsilon_2\bigg)\mathcal{F}(\Psi)(t)dt\nonumber \\
& = \frac{1}{(2\pi)^{d/2}}\int_{\mathbb{R}^d}e^{i\langle x_j-x_k,t\rangle}\bigg(\int_{\mathbb{R}^d}e^{i\langle \epsilon_1,t \rangle}\int_{\mathbb{R}^d} e^{i\langle -\epsilon_2,t \rangle}p(\epsilon_1)p(\epsilon_2)d\epsilon_1 d\epsilon_2\bigg)\mathcal{F}(\Psi)(t)dt\nonumber \\
& =  \frac{1}{(2\pi)^{d/2}}\int_{\mathbb{R}^d}e^{i\langle x_j-x_k,t\rangle}\bigg(\int_{\mathbb{R}^d} e^{i\langle \epsilon_1,t \rangle}p(\epsilon_1)d\epsilon_1\bigg)\bigg(\int_{\mathbb{R}^d}e^{i\langle -\epsilon_2,t \rangle}p(\epsilon_2)d\epsilon_2\bigg)\mathcal{F}(\Psi)(t)dt.
\end{align}
For any $w=(w_1,\ldots,w_n)^T$, by \eqref{eq:Psi1Fourier}, we have
\begin{align*}
& \sum_{j,k=1}^n w_j\bar w_k \Psi_S(x_j - x_k)\\
= & \sum_{j,k=1}^n w_j\bar w_k\frac{1}{(2\pi)^{d/2}}\int_{\mathbb{R}^d}e^{i\langle x_j-x_k,t\rangle} \bigg(\int_{\mathbb{R}^d} e^{i\langle \epsilon_1,t \rangle}p(\epsilon_1)d\epsilon_1\bigg)\bigg(\int_{\mathbb{R}^d}e^{i\langle -\epsilon_2,t \rangle}p(\epsilon_2)d\epsilon_2\bigg)\mathcal{F}(\Psi)(t)dt\\
% = & \frac{1}{(2\pi)^{d/2}}\int_{\mathbb{R}^d}\bigg|\sum_{j=1}^n w_j e^{i\langle x_j,t \rangle}\bigg|^2\bigg(1-\bigg(\int_{\mathbb{R}^d} e^{i\langle -\epsilon_1,t \rangle}p(\epsilon_1)d\epsilon_1\bigg)\bigg(\int_{\mathbb{R}^d}e^{i\langle \epsilon_2,t \rangle}p(\epsilon_2)d\epsilon_2\bigg)\bigg)\mathcal{F}(\Psi)(t)dt.
= & \frac{1}{(2\pi)^{d/2}}\int_{\mathbb{R}^d}\bigg|\sum_{j=1}^n w_j e^{i\langle x_j,t \rangle}\bigg|^2\bigg(\int_{\mathbb{R}^d} e^{i\langle \epsilon_1,t \rangle}p(\epsilon_1)d\epsilon_1\bigg)\bigg(\int_{\mathbb{R}^d}e^{i\langle -\epsilon_2,t \rangle}p(\epsilon_2)d\epsilon_2\bigg)\mathcal{F}(\Psi)(t)dt.
\end{align*}
Let
\begin{align*}
c(t) = \bigg(\int_{\mathbb{R}^d} e^{i\langle \epsilon_1,t \rangle}p(\epsilon_1)d\epsilon_1\bigg)\bigg(\int_{\mathbb{R}^d}e^{i\langle -\epsilon_2,t \rangle}p(\epsilon_2)d\epsilon_2\bigg).
\end{align*}
Thus, $c(t) \in \mathbb{R}$ and $c(t) > 0$. Therefore, $\sum_{j,k=1}^n w_j\bar w_k \Psi_S(x_j - x_k) \geq 0$, and equal to zero if and only if $w = 0$, which finishes the proof.

\section{Proof of Theorem \ref{propMSPEKALEN}}\label{App:pfThmkalen}

Consider the following Gaussian process with extrinsic error,
\begin{align}\label{eq:stomoding}
y_S(x) = M_S(x) + \delta(x),
\end{align}
where $M_S$ is a mean zero Gaussian process with covariance function $\sigma^2\Psi_S$, and $\delta(x)$ is an independent noise process with mean zero and variance $\mu$. The best linear unbiased predictor of \eqref{eq:stomoding} is 
\begin{align}\label{eq:BLUPSK}
\hat f_S(x) = r_N(x)^T(K_S + \mu I_n)^{-1}Y,
\end{align}
and the MSPE is 
\begin{align}\label{eq:stoMSPE}
{\rm MSPE}_S = \sigma^2\Psi_S(x-x) - r_N(x)^T(K_S + \mu I_n)^{-1}r_N(x).
\end{align}
{By Lemma \ref{lemma2skzero}}, \eqref{eq:stoMSPE} goes to zero as the fill distance of design points $X$ goes to zero.  

%A proof is provided in Appendix \ref{pfOfProp}.

Take $\mu = \sigma^2(\Psi(x-x) - \Psi_S(x-x))$. It can be seen that \eqref{eq:stoMSPE} is equal to $\sigma^2\Psi_S(x-x) - r_N(x)K^{-1}r_N(x)$. By \eqref{eq:MSPEKALEN}, $\mathbb{E}(y(x) - \hat{y}(x))^2 ={\rm MSPE}_S + \sigma^2(\Psi(x-x) - \Psi_S(x-x))$, which converges to $\sigma^2(\Psi(x-x) - \Psi_S(x-x))$ as the fill distance of the design points goes to zero. This completes the proof.

\section{Proof of Theorem \ref{Thm:main}}\label{App:pfThmmain}
Without loss of generality, assume $\sigma = 1$. First, we consider there is noise at an unobserved point. For any $u = (u_1,\ldots,u_n)^T$, it can be shown that the MSPE of predictor $u^TY$ is
\begin{align}\label{pfthm2eq1}
 & \mathbb{E}\bigg\|\Psi(\cdot- (x + \epsilon)) - \sum_{i=1}^n u_i\Psi(\cdot-( x_i + \epsilon))\bigg\|^2_{\mathcal{N}_{\Psi}}\nonumber\\   
 = & \Psi(x-x) - 2\sum_{j=1}^n u_j\Psi_S(x-x_j) +\sum_{j,k=1}^n u_j u_k\Psi_S(x_j - x_k) + a\|u\|_2^2,
\end{align}
where $\|\cdot\|_{\mathcal{N}_{\Psi}(\Omega)}$ is the norm of the reproducing kernel Hilbert space $\mathcal{N}_{\Psi}(\Omega)$ and $a =  \Psi(x-x) - \Psi_S(x-x)$.
Notice that 
\begin{align*}
\Psi_S(x_j - x_k) & = \frac{1}{(2\pi)^{d/2}}\int_{\mathbb{R}^d}e^{i\langle x_j-x_k,t \rangle}c(t)\mathcal{F}(\Psi)(t)dt,
\end{align*}
where
\begin{align*}
c(t) = \bigg(\int_{\mathbb{R}^d} e^{i\langle \epsilon_j,t \rangle}p(\epsilon_j)d\epsilon_j\bigg)\bigg(\int_{\mathbb{R}^d}e^{i\langle -\epsilon_k,t \rangle}p(\epsilon_k)d\epsilon_k\bigg).
\end{align*}
Since $|e^{i\langle -\epsilon_j,t \rangle}|\leq 1$, $c(t)\leq 1$. Therefore, \eqref{pfthm2eq1} can be bounded by
\begin{align}\label{pfthm2neq1}
&  \Psi(x-x) - 2\sum_{j=1}^n u_j\Psi_S(x-x_j) +\sum_{j,k=1}^n u_j u_k\Psi_S(x_j - x_k) + a\|u\|_2^2\nonumber\\
 = & u^T \Psi_S(X-X) u -2u^T \Psi_S(X-x) + \Psi_S(x-x) + a\|u\|_2^2\nonumber+a\nonumber\\
= & \frac{1}{(2\pi)^{d/2}}\int_{\mathbb{R}^d}\bigg|\sum_{j=1}^n u_j e^{i\langle x_j,t \rangle} - e^{i\langle x,t \rangle}\bigg|^2 c(t) \mathcal{F}(\Psi)(t)dt+ a\|u\|_2^2+a\nonumber\\
\leq & \frac{1}{(2\pi)^{d/2}}\int_{\mathbb{R}^d}\bigg|\sum_{j=1}^n u_j e^{i\langle x_j,t \rangle} - e^{i\langle x,t \rangle}\bigg|^2 \mathcal{F}(\Psi)(t)dt+ a\|u\|_2^2+a\nonumber\\
= &  u^T \Psi(X-X) u -2u^T \Psi(X-x) + \Psi(x-x) + a\|u\|_2^2+a\nonumber\\
\leq & \max\{1,a/\mu\}(u^T \Psi(X-X) u -2u^T \Psi(X-x) + \Psi(x-x) + \mu\|u\|_2^2)+a,
\end{align} 
where $ \Psi(X-x) =  \Psi(x-X)^T$. Plugging
\begin{align*}
u = (\Psi(X-X)+\mu I)^{-1}\Psi(X-x),
 \end{align*}
into \eqref{pfthm2eq1} and \eqref{pfthm2neq1}, we have the MSPE of predictor \eqref{eq:BLUPSK1} upper bounded by
\begin{align*}
\max\{1,a/\mu\}(\Psi(x-x) - \Psi(x-X)(\Psi(X-X)+\mu I)^{-1}\Psi(X-x))+a.
\end{align*}
By Lemma \ref{lemma2skzero}, $\Psi(x-x) - \Psi(x-X)(\Psi(X-X)+\mu I)^{-1}\Psi(X-x)$ converges to zero as the fill distance goes to zero since $\mu$ is a constant, which completes the proof in this case.

Next, we consider the case that there is no noise at an unobserved point. For any $u = (u_1,\ldots,u_n)^T$, it can be shown that the MSPE of predictor $u^TY$ in this case is
\begin{align}\label{pfthm2eq2}
 & \mathbb{E}\bigg\|\Psi(\cdot- x) - \sum_{j=1}^n u_j\Psi(\cdot-( x_j + \epsilon))\bigg\|^2_{\mathcal{N}_{\Psi}}\nonumber\\   
 = & u^T \Psi_S(X-X) u -2u^T r(x) + \Psi(x-x) + a\|u\|_2^2.
\end{align}
Let $b(t) = \int_{\mathbb{R}^d}e^{i\langle \epsilon_i,t \rangle}h(\epsilon_i)d\epsilon_i$. Thus, for any $u = (u_1,\ldots,u_n)^T$, we have
\begin{align}\label{pfthm2neq2}
& u^T \Psi_S(X-X) u -2u^T r(x) + \Psi(x-x) + a\|u\|_2^2\nonumber\\
= & \frac{1}{(2\pi)^{d/2}}\int_{\mathbb{R}^d}\bigg|\sum_{j=1}^n  u_j e^{i\langle x_j,t \rangle}b(t) - e^{i\langle x,t \rangle}\bigg|^2 \mathcal{F}(\Psi)(t)dt+ a\|u\|_2^2\nonumber\\
\leq & \frac{1 + C^{2}}{(2\pi)^{d/2}}\int_{\mathbb{R}^d}\bigg|\sum_{j=1}^n u_j e^{i\langle x_j,t \rangle} - e^{i\langle x,t \rangle}\bigg|^2 |b(t)|^2\mathcal{F}(\Psi)(t)dt  + \frac{1 + C^{-2}}{(2\pi)^{d/2}}\int_{\mathbb{R}^d} |1-|b(t)||^2\mathcal{F}(\Psi)(t)dt +  a\|u\|_2^2\nonumber\\
\leq & (1 + C^2)( u^T \Psi(x-x) u -2u^T \Psi(X-x) + \Psi(x-x)) + a\|u\|_2^2 + (1 + C^{-2})\frac{1}{(2\pi)^{d/2}}\int_{\mathbb{R}^d} |1-|b(t)||^2\mathcal{F}(\Psi)(t)dt \nonumber\\
\leq & \max\{(1 + C^2),a/\mu\}(u^T \Psi(X-X) u -2u^T \Psi(X-x) + \Psi(x-x) + \mu\|u\|_2^2) + \frac{1 + C^{-2}}{(2\pi)^{d/2}}\int_{\mathbb{R}^d} |1-|b(t)||^2\mathcal{F}(\Psi)(t)dt,
\end{align}
where we use $2\langle a,b\rangle\leq C^2|a|^2 + C^{-2}|b|^2$ in the first inequality, with $C$ a fixed constant. Plugging
\begin{align*}
u = (\Psi(X-X)+\mu I)^{-1}\Psi(X-x),
 \end{align*}
into \eqref{pfthm2eq2} and \eqref{pfthm2neq2}, we have the MSPE of predictor \eqref{eq:BLUPSK1} upper bounded by
\begin{align*}
\max\{(1 + C^2),a/\mu\}(\Psi(x-x) - \Psi(x-X)(\Psi(X-X)+\mu I)^{-1}\Psi(X-x))+\frac{1 + C^{-2}}{(2\pi)^{d/2}}\int_{\mathbb{R}^d} |1-|b(t)||^2\mathcal{F}(\Psi)(t)dt.
\end{align*}
By Lemma \ref{lemma2skzero}, $\Psi(x-x) - \Psi(x-X)(\Psi(X-X)+\mu I)^{-1}\Psi(X-x)$ converges to zero as the fill distance goes to zero since $\mu$ is a constant. { The constant $C$ influences the number of design points needed such that $\max\{(1 + C^2),a/\mu\}(\Psi(x-x) - \Psi(x-X)(\Psi(X-X)+\mu I)^{-1}\Psi(X-x))$ is close to zero. For a fixed number of design points, the larger $C$ is, the larger $\max\{(1 + C^2),a/\mu\}(\Psi(x-x) - \Psi(x-X)(\Psi(X-X)+\mu I)^{-1}\Psi(X-x))$ is. To derive an explicit bound, we let $C^2 = 25$, which yields an asymptotic upper bound 
\begin{align*}
    \frac{1.04}{(2\pi)^{d/2}}\int_{\mathbb{R}^d} |1-|b(t)||^2\mathcal{F}(\Psi)(t)dt.
\end{align*}
This finishes the proof.}

\section{Proof of Proposition \ref{Upprop}}\label{pfUpprop}
Notice that $\mathbb{E}(e^{i\epsilon_n^Tt})$ converges to $1$ since $\epsilon_n$ converges to $0$ in distribution and $e^{i\epsilon_n^Tt}$ is bounded, and $b(t)$ is bounded for all $t\in \mathbb{R}^d$. By dominated convergence theorem, the result holds.

\section{Proof of Theorem \ref{thmparaest}}\label{pfthmparaest}
We first present a lemma, which is a generalization of Lemma \ref{lemma2skzero}.
\begin{lemma}\label{lemma2skallzero}
Suppose the conditions of Theorem \ref{thmparaest} hold. Then we have $\Psi(x-x) - \hat \Psi(x-X)(\hat\Psi(X-X) + \hat\mu I)^{-1}\hat\Psi(x-X)^T$ converges to zero as the fill distance of $X$ converges to zero, where $\hat \Psi=\hat \Psi_1$ or $\hat \Psi_2$.
\end{lemma}
\begin{proof}
The proof of Lemma \ref{lemma2skallzero} is similar to the proof of Lemma \ref{lemma2skzero}. The only difference is that if we define $\hat g(t) = \hat \Psi(t-x) -\hat \Psi(t-X)(\hat \Psi(x-x) + \hat \mu I)^{-1}\hat \Psi(x-X)^T$, then $\|\hat g\|_{H^\nu(\Omega)}\leq C_2$ for all $\hat g$. Thus, the result follows from the proof of Lemma \ref{lemma2skzero}.
\end{proof}

Now we are ready to show the proof of Theorem \ref{thmparaest}. Let $\tilde y(x)$ be the stochastic Kriging predictor with estimated parameters $(\hat \theta_2,\hat \mu)$. Thus,
\begin{align}
    \tilde y(x) = \hat \Psi_2(x-X) (\hat \Psi_2(X-X) + \hat \mu I)^{-1}Y,
\end{align}
where $\hat \Psi_2 (x,X) = (\hat \Psi_2(x,x_1),\ldots, \hat \Psi_2(x,x_n))$ and $\hat{\Psi}_2(X,X)=(\hat \Psi_2(x_j - x_k))_{jk}$. 
% Because KALE and KALEN are best linear unbiased predictors, i.e., have the smallest MSPE among all linear predictors, and $\tilde y(x) $ is a linear predictor,
% We first show the statement is true for the stochastic Kriging predictor $\tilde y$.

\textit{Proof of Statement (i):}

Direct calculation shows that the MSPE can be expressed as 
\begin{align}\label{eqest1}
    \mathbb{E}(y(x) - \tilde y(x))^2 = & \sigma^2(\Psi(x-x) - 2\hat \Psi_2(x-X) (\hat \Psi_2(X-X) + \hat \mu I)^{-1}r_N(x) \nonumber\\
    & + \hat \Psi_2(x-X) (\hat \Psi_2(X-X)+ \hat \mu I)^{-1}K(\hat \Psi_2(X-X) + \hat \mu I)^{-1}\hat \Psi_2(x-X)^T),
\end{align}
where $K$ and $r_N$ are as in \eqref{eq:KernelMatrixNoi} and \eqref{eq:RvectorNoiX}, respectively. Similar to \eqref{pfthm2neq1}, we have for any $u=(u_1,\ldots,u_n)^T$,
\begin{align}\label{eqest2}
    & \Psi(x-x) - 2\sum_{j=1}^n u_j\Psi_S(x-x_j) +\sum_{j,k=1}^n u_j u_k\Psi_S(x_j - x_k) + a\|u\|_2^2\nonumber\\
 = & u^T \Psi_S(X-X) u -2u^T \Psi_S(X-x) + \Psi_S(x-x) + a\|u\|_2^2\nonumber+a\nonumber\\
= & \frac{1}{(2\pi)^{d/2}}\int_{\mathbb{R}^d}\bigg|\sum_{j=1}^n u_j e^{i\langle x_j,t \rangle} - e^{i\langle x,t \rangle}\bigg|^2 c(t) \mathcal{F}(\Psi)(t)dt+ a\|u\|_2^2+a\nonumber\\
\leq & \frac{1}{(2\pi)^{d/2}}\int_{\mathbb{R}^d}\bigg|\sum_{j=1}^n u_j e^{i\langle x_j,t \rangle} - e^{i\langle x,t \rangle}\bigg|^2 \mathcal{F}(\Psi)(t)dt+ a\|u\|_2^2+a\nonumber\\
\leq & \frac{A_1}{(2\pi)^{d/2}}\int_{\mathbb{R}^d}\bigg|\sum_{j=1}^n u_j e^{i\langle x_j,t \rangle} - e^{i\langle x,t \rangle}\bigg|^2 \mathcal{F}(\hat\Psi_2)(t)dt+ a\|u\|_2^2+a\nonumber\\
= &  A_1(u^T \hat \Psi_2(X-X) u -2u^T \hat \Psi_2(X-x) + \hat \Psi_2(x-x)) + a\|u\|_2^2+a\nonumber\\
\leq & \max\{A_1, a/\hat \mu\} (u^T \hat \Psi_2(X-X) u -2u^T \hat \Psi_2(X-x) + \hat \Psi_2(x-x) + \hat \mu\|u\|_2^2)+a,
\end{align}
where
\begin{align*}
c(t) = \bigg(\int_{\mathbb{R}^d} e^{i\langle \epsilon_j,t \rangle}p(\epsilon_j)d\epsilon_j\bigg)\bigg(\int_{\mathbb{R}^d}e^{i\langle -\epsilon_k,t \rangle}p(\epsilon_k)d\epsilon_k\bigg),
\end{align*}
and $a =  \Psi(x-x) - \Psi_S(x-x)$. Plugging
\begin{align*}
u = (\hat \Psi_2(X-X)+\hat \mu I)^{-1}\hat \Psi_2(X-x),
 \end{align*}
into \eqref{eqest1} and \eqref{eqest2}, we have the MSPE of predictor \eqref{eqest1} is upper bounded by
\begin{align*}
&  \max\{A_1, a/\hat \mu\}(\hat \Psi_2(x-x) - \hat \Psi_2(x-X)(\hat \Psi_2(X-X)+\hat \mu I)^{-1}\hat \Psi_2(X-x))+a\\
\leq &  \max\{A_1, aC\}(\hat \Psi_2(x-x) - \hat \Psi_2(x-X)(\hat \Psi_2(X-X)+ C I)^{-1}\hat \Psi_2(X-x))+a
\end{align*}
By Lemma \ref{lemma2skallzero}, $\hat \Psi_2(x-x) - \hat \Psi_2(x-X)(\hat \Psi_2(X-X)+C I)^{-1}\hat \Psi_2(X-x)$ converges to zero as the fill distance goes to zero, which indicates that $\sigma^2 a$ is an asymptotic upper bound on the MSPE of stochastic Kriging with estimated parameters. Note that $\sigma^2 a$ is also the limit of KALEN with the true parameters, which is the best linear unbiased predictor. Therefore, $\sigma^2 a$ is the limit of stochastic Kriging with estimated parameters. 

Note that KALEN is
\begin{align}\label{eq:pfKalenest}
    \hat y(x) = \hat r_N(x)^T(\hat \Psi_S(X-X) + \hat a I)^{-1}Y,
\end{align}
where $\hat \Psi_S(X-X) = (\hat \Psi_S(x_j-x_k))_{jk}$, $\hat r_N(x) = (\hat\Psi_S(x-x_1),...,\hat \Psi_S(x-x_n))$, 
\begin{align*}
\hat \Psi_S(s-t) = \iint\hat \Psi_1(s+\epsilon_1-(t+\epsilon_2))\hat p(\epsilon_1)\hat p(\epsilon_2)d\epsilon_1 d\epsilon_2,
\end{align*}
and $\hat a = \hat \Psi_1(x-x) - \hat \Psi_S(x-x)$. Condition (4) in Theorem \ref{thmparaest} implies that $\hat a$ is bounded away from zero. Thus, repeating the argument in the proof of stochastic Kriging completes the proof of Statement (i).

\textit{Proof of Statement (ii):}

By direct calculation, it can be shown that
\begin{align}\label{eqest3}
    \mathbb{E}(y(x) - \tilde y(x))^2 = & \sigma^2(\Psi(x-x) - 2\hat \Psi_2(x-X) (\hat \Psi_2(X-X) + \hat \mu I)^{-1}r(x) \nonumber\\
    & + \hat \Psi_2(x-X) (\hat \Psi_2(X-X)+ \hat \mu I)^{-1}K(\hat \Psi_2(X-X) + \hat \mu I)^{-1}\hat \Psi_2(X-x)),
\end{align} 
where $r(x)$ is as in \eqref{eq:RvectorNoi}. Let $b(t) = \int_{\mathbb{R}^d}e^{i\langle \epsilon_j,t \rangle}p(\epsilon_j)d\epsilon_j$. For any $u = (u_1,\ldots,u_n)^T$, we have
\begin{align}\label{pfpara2}
& u^T \Psi_S(X-X) u -2u^T r(x) + \Psi(x-x) + a\|u\|_2^2\nonumber\\
= & \frac{1}{(2\pi)^{d/2}}\int_{\mathbb{R}^d}\bigg|\sum_{j=1}^n  u_j e^{i\langle x_j,t \rangle}b(t) - e^{i\langle x,t \rangle}\bigg|^2 \mathcal{F}(\Psi)(t)dt+ a\|u\|_2^2\nonumber\\
\leq & \frac{(1 + C_1^2)}{(2\pi)^{d/2}}\int_{\mathbb{R}^d}\bigg|\sum_{j=1}^n u_j e^{i\langle x_j,t \rangle} - e^{i\langle x,t \rangle}\bigg|^2 |b(t)|^2\mathcal{F}(\Psi)(t)dt  + \frac{(1 + C_1^{-2})}{(2\pi)^{d/2}}\int_{\mathbb{R}^d} |1-|b(t)||^2\mathcal{F}(\Psi)(t)dt +  a\|u\|_2^2\nonumber\\
\leq & \frac{(1 + C_1^2)A_1}{(2\pi)^{d/2}}\int_{\mathbb{R}^d}\bigg|\sum_{j=1}^n u_j e^{i\langle x_j,t \rangle} - e^{i\langle x,t \rangle}\bigg|^2 |b(t)|^2\mathcal{F}(\hat \Psi_2)(t)dt  + \frac{(1 + C_1^{-2})}{(2\pi)^{d/2}}\int_{\mathbb{R}^d} |1-|b(t)||^2\mathcal{F}(\Psi)(t)dt +  a\|u\|_2^2\nonumber\\
\leq &  (1 + C_1^2)A_1(u^T \hat \Psi_2(X-X) u -2u^T \hat \Psi_2(X-x) + \hat \Psi_2(x-x)) + a\|u\|_2^2 + \frac{(1 + C_1^{-2})}{(2\pi)^{d/2}}\int_{\mathbb{R}^d} |1-|b(t)|^2|\mathcal{F}(\Psi)(t)dt \nonumber\\
\leq & \max\{(1 + C_1^2)A_1,a/\hat \mu\}(u^T \hat \Psi_2(X-X) u -2u^T \hat \Psi_2(X-x) + \hat \Psi_2(x-x) + \hat \mu\|u\|_2^2)\nonumber\\
&+ \frac{(1 + C_1^{-2})}{(2\pi)^{d/2}}\int_{\mathbb{R}^d} |1-|b(t)|^2|\mathcal{F}(\Psi)(t)dt.
\end{align}
Plugging $u = (\hat \Psi_2(X-X)+\hat \mu I)^{-1}\hat \Psi_2(X-x),$ into \eqref{eqest3} and \eqref{pfpara2}, we find the MSPE of predictor \eqref{eq:BLUPSK1} is upper bounded by
\begin{align*}
& \max\{(1 + C_1^2)A_1,a/\hat \mu\}(\hat \Psi_2(x-x) - \hat \Psi_2(x-X)(\hat \Psi_2(X-X)+\hat \mu I)^{-1}\hat \Psi_2(X-x))\nonumber\\
& + \frac{(1 + C_1^{-2})}{(2\pi)^{d/2}}\int_{\mathbb{R}^d} |1-|b(t)|^2|\mathcal{F}(\Psi)(t)dt\\
\leq & \max\{(1 + C_1^2)A_1,aC\}(\hat \Psi_2(x-x) - \hat \Psi_2(x-X)(\hat \Psi_2(X-X)+C I)^{-1}\hat \Psi_2(X-x))\nonumber\\
& + \frac{(1 + C_1^{-2})}{(2\pi)^{d/2}}\int_{\mathbb{R}^d} |1-|b(t)|^2|\mathcal{F}(\Psi)(t)dt.
\end{align*}
We take $C_1^2=25$. By Lemma \ref{lemma2skallzero}, $\hat\Psi_2(x-x) - \hat\Psi_2(x-X)(\hat\Psi_2(X-X)+C I)^{-1}\hat\Psi_2(X-x)$ converges to zero as the fill distance goes to zero since $C$ is a constant, which finishes the proof for stochastic Kriging.

Note that the KALE is 
\begin{align*}
    \hat f(x) = \hat r(x)^T(\hat \Psi_S(X-X) + \hat a I)^{-1}Y,
\end{align*}
where $\hat r(x)$ is as in \eqref{eq:RvectorNoi} with estimated parameters, and $\hat \Psi_S(X-X)$ and $\hat a$ are as in \eqref{eq:pfKalenest}. Because $\hat \Psi_1$ is a correlation function and $\hat p(\cdot) = p(\cdot)$, we have $ \hat \Psi_1(x-x) = \Psi(x-x)$ and $\hat \Psi_S(x-x) = \Psi_S(x-x)$, which imply $\hat a = \hat \Psi(x-x) - \hat \Psi_S(x-x) = \Psi(x-x) - \Psi_S(x-x) = a$. Then for any $u = (u_1,\ldots,u_n)^T$, we have
\begin{align}\label{pfpara22}
& u^T \Psi_S(X-X) u -2u^T r(x) + \Psi(x-x) + a\|u\|_2^2\nonumber\\
= & \frac{1}{(2\pi)^{d/2}}\int_{\mathbb{R}^d}\bigg|\sum_{j=1}^n  u_j e^{i\langle x_j,t \rangle}b(t) - e^{i\langle x,t \rangle}\bigg|^2 \mathcal{F}(\Psi)(t)dt+ a\|u\|_2^2\nonumber\\
\leq & \frac{A_1}{(2\pi)^{d/2}}\int_{\mathbb{R}^d}\bigg|\sum_{j=1}^n  u_j e^{i\langle x_j,t \rangle}b(t) - e^{i\langle x,t \rangle}\bigg|^2 \mathcal{F}(\hat \Psi_1)(t)dt+ a\|u\|_2^2\nonumber\\
= & A_1(u^T\hat \Psi_S(X-X) u -2u^T \hat r(x) + \hat \Psi_1(x-x)) + a\|u\|_2^2.
\end{align}
Note that $\hat f(x)$ minimizes \eqref{pfpara22}. Then repeating the proof of Theorem \ref{Thm:main} gives an upper bound
\begin{align*}
    \frac{1.04A_1\sigma^2}{(2\pi)^{d/2}}\int_{\mathbb{R}^d} |1-|b(t)|^2|\mathcal{F}(\hat \Psi_1)(t)dt.
\end{align*}
Together with $\mathcal{F}(\hat \Psi_1)(t)\leq A_2 \mathcal{F}(\Psi)(t)$ for any $t$, we finish the proof.

\bibliographystyle{apalike}
\bibliography{paperref}

\end{document}